\documentclass[twoside,11pt]{amsart}

\usepackage{amsmath,latexsym,amssymb,mathptm, times,verbatim, enumerate}
    \usepackage{longtable}

\input amssym.def
\input amssym
\input xypic
\input xy
\xyoption{all}
\newtheorem{theorem}{Theorem}[section]
\newtheorem{lemma}[theorem]{Lemma}

\newtheorem{proposition}[theorem]{Proposition}

\newtheorem{observation}[theorem]{Observation}

\theoremstyle{definition}
\newtheorem{definition}[theorem]{Definition}
\newtheorem{example}[theorem]{Example}
\newtheorem{remark}[theorem]{Remark}
\newtheorem{remarks}[theorem]{Remarks}
\newtheorem{data}[theorem]{Data}

\newtheorem{conventions}[theorem]{Conventions}

\newtheorem{summary}[theorem]{Summary}
\newtheorem{chunk}[theorem]{}
\newtheorem*{proof-free}{Proof of Lemma \ref{main-lemma}}

\numberwithin{equation}{theorem}

\begin{document}

\baselineskip=16pt

\title[The explicit minimal resolution constructed from a Macaulay inverse system]{\bf The explicit minimal resolution constructed from\\ a Macaulay inverse system}
\date{\today}

\author[Sabine El Khoury and Andrew R. Kustin]
{Sabine El Khoury and Andrew R. Kustin}

\thanks{AMS 2010 {\em Mathematics Subject Classification}.
Primary 13H10, 13E10, 13D02, 13A02.}

\thanks{The second author was partially supported by the Simons Foundation.}

\thanks{Keywords: Artinian rings, Buchsbaum-Eisenbud ideals, Build resolution directly from inverse system,    Gorenstein rings, Linear presentation, Linear resolution, Macaulay inverse system,   Pfaffians, Resolutions }

\address{Mathematics Department,
American University of Beirut,
Riad el Solh 11-0236,
Beirut,
Lebanon}
\email{se24\@aub.edu.lb}

\address{Department of Mathematics, University of South Carolina,
Columbia, SC 29208} \email{kustin@math.sc.edu}

\begin{abstract} Let $A$ be a standard-graded Artinian Gorenstein algebra of embedding codimension three over a field $\pmb k$. In the generic case, the minimal homogeneous resolution, ${\mathbb G}$, of $A$, by free $\operatorname{Sym}_{\bullet}^{\pmb k}(A_1)$ modules, is Gorenstein-linear.
Fix a basis $x,y,z$ for the ${\pmb k}$-vector space $A_1$. If ${\mathbb G}$ is Gorenstein linear, then the socle degree of $A$ is necessarily even, and, if $n$ is the least index with $\dim_{\pmb k} A_n$ less than $\dim_{\pmb k} {\operatorname{Sym}}_{n}^{\pmb k}(A_1)$, then the socle degree of $A$ is $2n-2$.  Let $$\Phi=\sum\alpha_m m^*,$$as $m$ roams over the monomials in $x,y,z$ of degree $2n-2$, with $\alpha_m\in \pmb k$,  be an arbitrary homogeneous element  of degree $2n-2$ in the divided power module $D_{\bullet}^{\pmb k}(A_1^*)$. 
 The annihilator of $\Phi$ (denoted $\operatorname{ann}\Phi$) is the ideal of elements $f$ in ${\operatorname{Sym}}_{\bullet}^{\pmb k}(A_1)$ with $f(\Phi)=0$. The element  $\Phi$ of $D_{\bullet}^{\pmb k}(A_1^*)$ is the  Macaulay inverse system for the ring ${\operatorname{Sym}}_{\bullet}^{\pmb k}(A_1)/\operatorname{ann}\Phi$, which is necessarily Gorenstein and Artinian.
Consider the matrix $(\alpha_{mm'})$, as $m$ and $m'$ roam over the monomials in $x,y,z$ of degree $n-1$.  The ring ${\operatorname{Sym}}_{\bullet}^{\pmb k}(A_1)/\operatorname{ann}\Phi$
has a Gorenstein-linear resolution if and only if $\det (\alpha_{mm'})\neq 0$.  If $\det (\alpha_{mm'})\neq 0$, then we give explicit formulas for the minimal homogeneous resolution of  ${\operatorname{Sym}}_{\bullet}^{\pmb k}(A_1)/\operatorname{ann}\Phi$ in terms the $\alpha_m$'s and $x,y,z$.
\end{abstract}

\maketitle

For the time being, let $U$ be a vector space of dimension $d$ over a field $\pmb k$,
 $S=\operatorname{Sym}_{\bullet}^{\pmb k}U$ be a standard-graded polynomial ring in $d$ variables over  $\pmb k$, and $D=D_{\bullet}^{\pmb k}(U^*)$ be  
the graded $S$-module of graded $\pmb k$-linear homomorphisms from $S$ to $\pmb k$. In his 1916 paper \cite{M16}, Macaulay proved that
each element $\Phi$ of $D$ determines (in our language) an 
 Artinian Gorenstein ring $A_{\Phi}=S/\operatorname{ann}(\Phi)$; furthermore, each Artinian Gorenstein quotient of $S$ is obtained in this manner. Of course, $\Phi$ determines everything about the quotient $A_{\Phi}$; so in particular,  when $\Phi$ is a homogeneous element of $D$, then $\Phi$ determines a minimal resolution of $A_{\Phi}$ by free $S$-modules. The standard way to find this minimal resolution is to first solve some equations in order to determine a minimal  generating set for $\operatorname{ann}(\Phi)$ and then to use Gr\"obner basis techniques in order to find a minimal resolution of $A_{\Phi}$ by free $S$-modules. We are interested in by-passing all of the intermediate steps. We aim to describe a minimal  resolution of $A_{\Phi}$ directly (and in a polynomial manner) in terms of  the coefficients of $\Phi$, at least in the generic case. In \cite{EKK}, we proved that if $\Phi$ is homogeneous of even degree $2n-2$ and the pairing \begin{equation}\label{pp}S_{n-1}\times S_{n-1}\to \pmb k,\end{equation} given by
$(f,g)\mapsto fg(\Phi)$, is perfect, then a minimal resolution for $A_{\Phi}$ may be read directly, and in a polynomial manner, from the coefficients of $\Phi$. Furthermore, there is one such resolution for each pair $(d,n)$. Please notice that the pairing (\ref{pp}) 
is perfect if and only if the determinant
of the matrix $((m_im_j)\Phi)$, (as $m_i$ and $m_j$ roam over the monomials in $S$ of degree $n-1$), is non-zero. This is an open condition on the coefficients of 
$\Phi$ (which are precisely the values of $m\Phi$ as $m$ roams over the monomials of $S$ of degree $2n-2$); hence 
the pairing (\ref{pp})
is perfect whenever $\Phi$ is chosen generically. Furthermore, the pairing (\ref{pp}) is perfect if and only if the minimal resolution of $A_{\Phi}$ by free $S$-modules is Gorenstein-linear.

The paper \cite{EKK} proves the existence of a unique generic Gorenstein-linear resolution for each pair $(d,n)$; but exhibits this resolution only for the pair $(d,n)=(3,2)$. In the present paper, 
we exhibit this resolution when $d=3$ and $n\ge 2$ is arbitrary. Indeed,
once $n\ge 2$ is fixed, we exhibit an explicit complex $(\mathbb B,b)$ (see Definition~\ref{main-def} or Observation~\ref{b2} or Proposition~\ref{explc}, depending upon your tolerance for, and/or need to see, explicitness). If 
$U$ is a vector space over $\pmb k$ of dimension $d=3$ and $\Phi$ is a generic element of $D_{2n-2}^{\pmb k}(U^*)$, then $S\otimes \mathbb B$ is the minimal resolution of $A_{\Phi}$ by free $S=\operatorname{Sym}_\bullet^{\pmb k}(U)$ modules.

We preview the complex $\mathbb B$. (Complete details are given in Section~\ref{coord-free}.) This complex is built over $\mathbb Z$. Let $U$ be a free $\mathbb Z$-module of rank $3$ and $\mathfrak R$ be the ring  $\operatorname{Sym}_{\bullet}^{\mathbb Z}(U\oplus \operatorname{Sym}_{2n-2}^{\mathbb Z}U)$. (No harm is done by choosing bases $x,y,z$ for $U$ and $\{t_m\mid \text{$m$ is a monomial of degree $2n-2$ in $x$, $y$, $z$}\}$ for $\operatorname{Sym}_{2n-2}^{\mathbb Z}U$ and viewing $\mathfrak R$ as the polynomial ring $\mathbb Z[x,y,z,\{t_m\}]$; although we will not officially choose these bases until Section~\ref{mat-desc}. Until Section~\ref{mat-desc}, we will keep the calculation as coordinate-free as possible.) In the complex $\mathbb B$, one of the basis elements of $U$ (we call this element $x$) is given a distinguished role. The complex $\mathbb B$ is symmetric in the complementary basis elements (we call these elements $y$ and $z$) of $U$. Let $U_0$ be the free $\mathbb Z$-summand $\mathbb Z y\oplus \mathbb Z z$ of $U=\mathbb Zx\oplus \mathbb Zy\oplus \mathbb Z z$. The complex $\mathbb B$ is
$$\xymatrix{0\ar[r]&\mathfrak R\ar[r]& B_2\ar[r]^{b_2}&B_2^*\ar[r]&\mathfrak R,}$$
with $B_2=\mathfrak R\otimes_{\mathbb Z} (\operatorname{Sym}_{n-1}^{\mathbb Z}U_0 \oplus D_{n}^{\mathbb Z}(U_0^*))$. The presenting matrix $b_2$ is induced by an $\mathfrak R$-module homomorphism
$$\mathfrak b: {\textstyle \bigwedge^2_{{\mathfrak R}}} B_2\to {\mathfrak R}$$with
$$
\mathfrak b((\mu+\nu)\wedge (\mu'+\nu'))=\begin{cases}\phantom{+}x\cdot [\beta_3(\mu\wedge \mu')+\beta_2(\mu\otimes\nu')-\beta_2(\mu'\otimes\nu)+\beta_1(\nu\wedge \nu')]\vspace{5pt}\\+y\cdot \pmb \delta\cdot ( [z\mu](\nu')-[z\mu'](\nu))-z\cdot \pmb \delta\cdot ([y\mu](\nu')-[y\mu'](\nu)),\end{cases}$$
for $\mu,\mu'$ in ${\operatorname{Sym}}_{n-1}^{\mathbb Z}U_0$, $\nu,\nu'$ in $D_n^{\mathbb Z}(U_0^*)$, where the $\mathbb Z$-module homomorphisms $$\begin{array}{rlll}
\beta_1:&\bigwedge^2_{{\mathbb Z}}( D_{n}^{\mathbb Z}(U_0^*)) &\to& {\mathfrak R}_{0,*},\vspace{5pt}\\
\beta_2:&{\operatorname{Sym}}_{n-1}^{\mathbb Z}U_0\otimes_{{\mathbb Z}} D_{n}^{\mathbb Z}(U_0^*) &\to& {\mathfrak R}_{0,*}, \text{ and}\vspace{5pt}\\
\beta_3:& \bigwedge^2_{{\mathbb Z}}( {\operatorname{Sym}}_{n-1}^{\mathbb Z}U_0) &\to& {\mathfrak R}_{0,*}\end{array}$$ all are defined in terms of the classical adjoint of the map
\begin{equation}\label{him}{\operatorname{Sym}}_{n-1}^{\mathbb Z}U \to \mathfrak R_{0,*}\otimes_{\mathbb Z} D_{n-1}^{\mathbb Z}(U^*),\end{equation}
which is induced by the $\mathbb Z$-analogue of  (\ref{pp}), the determinant, $\pmb\delta$, of (\ref{him}), and the element 
${\widetilde{\Phi}}$  of $D_{2n-1}^{\mathbb Z}(U^*)$ with $x({\widetilde{\Phi}})=\Phi$ in $D_{2n-2}^{\mathbb Z}(U^*)$ and $\mu({\widetilde{\Phi}})=0$ for all $\mu\in {\operatorname{Sym}}_{2n-1}^{\mathbb Z}(U_0)$. (The ring $\mathfrak R_{0,*}$ is the subring 
$\operatorname{Sym}_{\bullet}^{\mathbb Z}(\operatorname{Sym}_{2n-2}^{\mathbb Z}U)$
of $\operatorname{Sym}_{\bullet}^{\mathbb Z}(U\oplus \operatorname{Sym}_{2n-2}^{\mathbb Z}U)=\mathfrak R$.)
We remark that the homomorphism $b_2$, which presents a generic grade $3$ Gorenstein ideal, is an alternating homomorphism, as is predicted by Buchsbaum and Eisenbud \cite{BE77}. However, our calculations never used the Buchsbaum-Eisenbud Theorem; and therefore, they provide an alternate proof of the Buchsbaum-Eisenbud Theorem for linearly presented grade three Gorenstein ideals. 

Section~\ref{conv} contains the conventions and notation that are used in the paper. Section~\ref{coord-free} is a complete and careful description of the maps and modules of $(\mathbb B,b)$. In Section~\ref{EKK} we recall the relevant results from \cite{EKK}. In particular, the complex $(\mathbb G,g)$ of Theorem~\ref{old-mvd} has all of the desired properties, except, one is not able to answer the (very basic) questions, ``What \underline{exactly} are $G_1$, $G_2$, $g_1$, and $g_2$?''
The main result of the present paper is that the explicitly constructed complex $(\mathbb B,b)$ has all of the properties of the complex $(\mathbb G,g)$. This result is stated as Theorem~\ref{main} and/or 
Lemma~\ref{main-lemma}.\ref{main-lemma-g}. The proof of Theorem~\ref{main} is carried out in Section~\ref{main-theorem}. Ultimately, in Lemma~\ref{main-lemma}.\ref{main-lemma-f}, we produce an isomorphism of complexes $\tau:(\mathbb B,b)\to (\mathbb E,e)$, where $(\mathbb E,e)$ is a sub-complex of $(\mathbb G,g)$ with 
$(\mathbb E,e)_{\pmb \delta}=(\mathbb G,g)_{\pmb \delta}$. In Section~\ref{mat-desc} we describe the homomorphisms of $(\mathbb B,b)$ in terms of the elements of the bi-graded polynomial ring $\mathbb Z[x,y,z,\{t_m\}]$ explicitly; the word ``induced'' does not appear in the section. The bi-homogeneous form of $(\mathbb B, b)$ is given just before Remark~\ref{10.1}. Section~\ref{ex} contains some explicit specializations of the generic complex $(\mathbb B,b)$; these examples are related to the project~\ref{P4}, which is described below.

There are numerous interesting projects which are related to the present project. We hope that the techniques and insights from the present project will lead to progress on these related projects.

\begin{chunk} We would like to find an explicit version of the resolution of \cite{EKK} for all values of $d=\operatorname{rank}_{\mathbb Z} U$. Indeed, an explicit version of this resolution would be very interesting even  when $d=4$. This resolution is a complete structure theorem for codimension four Artinian Gorenstein rings with Gorenstein-linear resolutions. It would be nice to know what the resolution  is in addition to knowing that the resolution  exists. Furthermore, we wonder how Miles Reid captures this resolution in his program \cite{R} for resolving  codimension four Gorenstein quotients.\end{chunk}

\begin{chunk}\label{P2} Can we prove a version of \cite[Thm.~6.15]{EKK} with the hypothesis that $A_{\Phi}$ has a Gorenstein-linear resolution replaced by the hypothesis that $A_{\Phi}$ is a compressed algebra? Tony Iarrobino \cite{I84} initiated the use of the word ``compressed'' to describe 
Artinian Gorenstein $\pmb k$-algebras which have the largest total length among all Artinian Gorenstein $\pmb k$-algebras with the specified embedding codimension and the specified socle degree. The set of Artinian Gorenstein algebras with a linear resolution is a proper subset of the set of compressed Artinian Gorenstein algebras. In some sense, this question
asks for the generic resolution of standard-graded Artinian Gorenstein algebras with odd socle degree.
Marilina Rossi and Liana \c{S}ega have recently written a remarkable paper \cite{RS} in which they prove that the Poincar\'e series of every finitely generated module over every local Artinian Gorenstein compressed algebra is rational provided the socle degree is not $3$. (The restriction on the socle degree is clearly needed because B{\o}gvad's examples \cite{Bo} of  Gorenstein rings with  transcendental Poincar\'e series  are compressed Artinian rings with socle degree equal to $3$). The Rossi-\c{S}ega theorem is especially remarkable because  so many of the usual tools for proving that Poincar\'e series are rational are not available to them. In particular, they do not know the minimal $R$-resolution of $R/I$ and they do not know if the minimal $R$-resolution of $R/I$ is an associative   DG-algebra. Rossi and \c{S}ega independently suggested to us that the project~\ref{P2} might be a plausible generalization of \cite[Thm.~6.15]{EKK}.\end{chunk}

\begin{chunk} Lucho Avramov asked  us ``Is the resolution of \cite{EKK} an associative DG-algebra?'' (This question is interesting only when $d\ge 5$.) If the answer is yes, it would help explain (and possibly simplify the proof of) the  Rossi-\c{S}ega Theorem. Our present thinking is that it might be possible to record such a pretty version of this resolution, for all values of $d$, that explicit formulas for multiplication on the resolution can be given.\end{chunk}

\begin{chunk}\label{P4} What is the 
the orbit space of $\operatorname{GL}_3\pmb k\times \operatorname{GL}_{2n+1}\pmb k$ acting on the space of $(2n+1)\times (2n+1)$ alternating matrices with homogeneous linear entries from the ring $\pmb k[x,y,z]$? This is the question which lead to \cite{EKK} and further comments about this question are contained in \cite{EKK}. Also, we return to this question in Section~\ref{ex}.
This question is of interest because there is much recent work concerning  the  equations
that define the Rees algebra of   ideals which are primary to the maximal ideal; see, for example, 
\cite{HSV,CHW,Bu,CKPU}. The driving force behind this work is the desire to understand the singularities of parameterized curves or surfaces; see \cite{SCG,CWL,BB,Bot,BD'A} and especially \cite{CKPU}. One of the key steps in \cite{CKPU} is the decomposition of the space of $3\times 2$ matrices with homogeneous entries from $\pmb k[x,y]$ of a fixed degree into disjoint orbits under the action of $\operatorname{GL}_3 \pmb k\times \operatorname{GL}_2\pmb k$. A successful answer to question (\ref{P4}) would have an immediate interpretation in terms of the defining equations of Rees algebras. Eventually, the Rees algebra result would have an interpretation in terms of singularities on parameterized surfaces.\end{chunk}

\section{Conventions}\label{conv}

If $R$ is a graded ring, then a homogeneous   complex of free $R$-modules  is  {\it Gorenstein-linear} if it has the form 
$$0\to R(-2n-t+2)^{s_{t}}\xrightarrow{\ d_{t}\ } R(-n-t+2)^{s_{t-1}}\xrightarrow{\ d_{t-1}\ }\dots\xrightarrow{\ d_{3}\ } R(-n-1)^{s_2}\xrightarrow{\ d_{2}\ } R(-n)^{s_1}\xrightarrow{\ d_{1}\ }R^{s_0},$$
for some integers $n$, $t$, and $s_i$. In other words, all of  the entries in all of the matrices $d_i$, except the first matrix and the last matrix, are homogeneous linear forms; and all of the entries in the first and last matrices are homogeneous forms of the same degree.

A graded ring $R=\bigoplus_{0\le i}R_i$ is called {\it standard-graded over} $R_0$, if $R$ is generated as an $R_0$-algebra by $R_1$ and $R_1$ is finitely generated as an $R_0$-module.

\begin{conventions}\label{conventions}${}$Let $U$ be a free $\mathbb Z$-module of rank $3$ and let $x$, $y$, $z$ be a basis for $U$.
\begin{enumerate}[\rm(a)]
\item For any set of variables $\{x_1,\dots,x_r\}$ and any degree $s$, we write
$\binom{x_1,\dots,x_r} s$ for   the set of monomials of degree $s$ in the variables $x_1,\dots,x_r$.

\item  We always  think of $x_1$ as $x$, $x_2$ as $y$, and $x_3$ as $z$.

\item If $S$ is a statement then 
$$\chi(S)=\begin{cases} 1,&\text{if $S$ is true,}\vspace{5pt}\\0,&\text{if $S$ is false.}\end{cases}$$

\item If $m$ is a monomial in the variables $x,y,z$, then $x|m$ is the statement ``$x$ divides $m$''.

\item We make much use of the fact that $D_{\bullet}^{\mathbb Z}(U^*)$ is a ${\operatorname{Sym}}_{\bullet}^{{\mathbb Z}}U$-module. In particular, if $\mu$ and $\mu'$ are in the ring ${\operatorname{Sym}}_{\bullet}^{{\mathbb Z}}U$ and $\nu$ is in the module  $D_{\bullet}^{\mathbb Z}(U^*)$, then
\begin{equation}\label{mod-act}\mu(\mu'(\nu))=(\mu\mu')(\nu)=(\mu'\mu)(\nu) = \mu'(\mu(\nu))\quad \in D_{\bullet}^{\mathbb Z}(U^*).\end{equation}
If $\mu\in {\operatorname{Sym}}_{i}^{{\mathbb Z}}U$ and $\nu\in D_{j}^{\mathbb Z}(U^*)$, then
$$\mu(\nu)\in D_{j-i}^{\mathbb Z}(U^*).$$
Furthermore, if $\mu$ and $\nu$ are homogeneous of the same degree, then \begin{equation}\label{flip-flop}\mu(\nu)=\nu(\mu)\quad \in {\mathbb Z}.\end{equation}

\item\label{conv-f} If $m$ is the monomial $x^ay^bz^c$ of ${\operatorname{Sym}}_{N}^{\mathbb Z}U$, then $m^*$ is defined to be the element ${x^*}^{(a)}{y^*}^{(b)}{z^*}^{(c)}$ of $D_{N}^{{\mathbb Z}}(U^*)$. The module action of ${\operatorname{Sym}}_{\bullet}^{{\mathbb Z}}U$ on $D_{\bullet}^{{\mathbb Z}}(U^*)$ makes 
$\{ m^*\mid m\in \binom{x,y,z}{N}\}$ be the ${{\mathbb Z}}$-module basis for the free ${\mathbb Z}$-module $D_{N}^{{\mathbb Z}}(U^*)$ which is dual to the ${{\mathbb Z}}$-module basis $\binom{x,y,z}{N}$ of ${\operatorname{Sym}}_{N}^{{\mathbb Z}}U$. (More information about divided power modules may be found, for example,  in \cite[subsect.~1.3]{EKK}.) Notice that if $m\in \binom{x,y,z}{N}$ for some $N$ and $x_i\in \{x,y,z\}$, then the module action of ${\operatorname{Sym}}_\bullet^{{\mathbb Z}} U$ on $
D_{\bullet}^{{\mathbb Z}}(U^*)$ gives \begin{equation}\label{xim} \textstyle x_i(m^*)=\chi(x_i|m)(\frac{m}{x_i})^*\in D_{N-1}^{{\mathbb Z}}(U^*).\end{equation}

\item In the present paper we have no need to consider the algebra structure of the Divided Power Algebra $D_{\bullet}^{{\mathbb Z}}(U^*)$; so, in particular, other than in (\ref{conv-f}) above, we will never write $\nu\nu'$ with $\nu$ and $\nu'$ in $D_{\bullet}^{{\mathbb Z}}(U^*)$. However, we will often write $(xm)^*\in D_{N+1}^{\mathbb Z}(U^*)$ for some monomial $m\in \binom{x,y,z}{N}$. Notice that $$x((xm)^*)=m\quad \text{and}\quad y((xm)^*)=\chi(y|m)\textstyle (x\frac my)^* \quad\text{in }D_{N}^{{\mathbb Z}}(U^*).$$

\end{enumerate}
\end{conventions}

\begin{summary}
Let $A$ be a standard-graded, Artinian, Gorenstein algebra over a field $\pmb k$. In the generic case, the minimal homogeneous resolution, ${\mathbb G}$, of $A$ by free $\operatorname{Sym}_{\bullet}^{\pmb k}(A_1)$ modules is Gorenstein-linear.
Fix a basis $x,y,z$ for the ${\pmb k}$-vector space $A_1$. If ${\mathbb G}$ is Gorenstein-linear, then the socle degree of $A$ is necessarily even, and, if $n$ is the least index with $\dim_{\pmb k} A_n<\dim_{\pmb k} {\operatorname{Sym}}_{n}^{\pmb k}(A_1)$, then the socle degree of $A$ is $2n-2$. (See, for example, \cite[Prop.~1.8]{EKK}.) Let $$\Phi=\sum_{m\in \binom{x,y,z}{2n-2}} \alpha_m m^*$$ be an arbitrary homogeneous element of the divided power module $D_{\bullet}^{\pmb k}(A_1^*)$ of degree $2n-2$. 
 The annihilator of $\Phi$ (denoted $\operatorname{ann}\Phi$) is the ideal of elements $f$ in ${\operatorname{Sym}}_{\bullet}^{\pmb k}(A_1)$ with $f(\Phi)=0$. The element  $\Phi$ of $D_{\bullet}^{\pmb k}(A_1^*)$ is the {\it Macaulay inverse system} for the ring ${\operatorname{Sym}}_{\bullet}^{\pmb k}(A_1)/\operatorname{ann}\Phi$, which is necessarily Gorenstein and Artinian.
Fix an order for the set $\binom{x,y,z}{n-1}$ and consider the matrix $(\alpha_{mm'})$ as $m$ and $m'$ roam over $\binom{x,y,z}{n-1}$ in the prescribed order. The ring ${\operatorname{Sym}}_{\bullet}^{\pmb k}(A_1)/\operatorname{ann}\Phi$
has a Gorenstein-linear resolution if and only if $\det (\alpha_{mm'})\neq 0$. (See, for example \cite[Prop.~1.8]{EKK}.) If $\det (\alpha_{mm'})\neq 0$, then in Theorem \ref{main} (see also, Proposition~\ref{explc} and Section~\ref{ex}) we give explicit formulas for the minimal resolution of  ${\operatorname{Sym}}_{\bullet}^{\pmb k}(A_1)/\operatorname{ann}\Phi$ in terms the $\alpha_m$'s and $x,y,z$.\end{summary}

\section{The coordinate-free version of $\mathbb B$.}\label{coord-free}

The main results of this paper concern the sequence of homomorphisms that we call $(\mathbb B,b)$. These homomorphisms are defined in a coordinate free manner in Definition~\ref{main-def}. A more explicit version of $(\mathbb B,b)$ is given in Proposition~\ref{explc}. 

The basic data is  given in \ref{Opening-Data}. All of $(\mathbb B,b)$  is made out of  Data \ref{Opening-Data}. There are two intermediate steps, Data \ref{manuf-data} and Data \ref{data-3}, where various maps and elements are created using the basic data of \ref{Opening-Data}, before the coordinate-free version of $(\mathbb B,b)$ is given in Definition \ref{main-def}. 
\begin{data}\label{Opening-Data} Let $U$ be a free $\mathbb Z$-module of rank $3$ and $n\ge 2$ be an integer.  

\begin{enumerate}[(a)]
\item Define ${\mathfrak R}$ to be the bi-graded ring
${\mathfrak R}={\operatorname{Sym}}_{\bullet}^{\mathbb Z}(U\oplus {\operatorname{Sym}}_{2n-2}^{\mathbb Z}U)$, where $$U\oplus 0={\mathfrak R}_{(1,0)}\quad\text{and}\quad 0\oplus{\operatorname{Sym}}_{2n-2}^{\mathbb Z}U={\mathfrak R}_{(0,1)}.$$ 

\item Define $\Psi:{\operatorname{Sym}}_{\bullet}^{\mathbb Z}U\to {\mathfrak R}$ to be the $\mathbb Z$-algebra homomorphism which is induced by the inclusion 
$$U={\mathfrak R}_{(1,0)}\subseteq {\mathfrak R}.$$ 

\item Define $\Phi: {\operatorname{Sym}}_{2n-2}^{\mathbb Z}U\to {\mathfrak R}$ to be the $\mathbb Z$-module homomorphism 
$${\operatorname{Sym}}_{2n-2}^{\mathbb Z}U=R_{(0,1)}\subseteq {\mathfrak R}.$$
\end{enumerate}
\end{data}

\begin{remark}\label{Phi} We may think of $\Phi$ as an element of ${\mathfrak R}\otimes_{\mathbb Z} D_{2n-2}^{\mathbb Z}(U^*)$. Indeed, if $(\{m_i\},\{m_i^*\})$ is any pair of dual bases for 
${\operatorname{Sym}}_{2n-2}^{\mathbb Z} U$ and $D_{2n-2}^{\mathbb Z}(U^*)$, respectively, then $\Phi$ and $\sum_{i} \Phi(m_i)\otimes m_i^*$ represent the same $\mathbb Z$-module homomorphism ${\operatorname{Sym}}_{2n-2}^{\mathbb Z}U\to {\mathfrak R}$. \end{remark}

The first collection of objects that we manufacture using the data from \ref{Opening-Data} all involve the symmetric pairing $${\operatorname{Sym}}_{n-1}^{\mathbb Z} U\otimes_{\mathbb Z} {\operatorname{Sym}}_{n-1}^{\mathbb Z} U
\xrightarrow{\text{multiplication}}
{\operatorname{Sym}}_{2n-2}^{\mathbb Z} U\stackrel{\Phi}{\longrightarrow} {\mathfrak R}.$$

\begin{data}\label{manuf-data} Retain the data of \ref{Opening-Data}. Let  $\mathrm{top}$ represent $\binom{n+1}2$, which is the rank of the free $\mathbb Z$-module ${\operatorname{Sym}}_{n-1}^{\mathbb Z}U$, and let $\Theta$ be a basis element for the rank one free $\mathbb Z$-module $\bigwedge^\mathrm{top}_{\mathbb Z} ({\operatorname{Sym}}_{n-1}^{\mathbb Z}U)$.
\begin{enumerate}[(a)]
\item\label{8.19.d} Define $\mathbb Z$-module homomorphisms $$P:{\operatorname{Sym}}_2^{\mathbb Z}({\operatorname{Sym}}_{n-1}^{\mathbb Z} U)\to {\mathfrak R}\quad \text{and} \quad p:{\operatorname{Sym}}_{n-1}^{\mathbb Z}U\to {\mathfrak R}\otimes_{\mathbb Z} D_{n-1}^{\mathbb Z} (U^*)$$ by
\begin{equation}\label{Pp}P(\mu\otimes\mu')=\Phi(\mu\mu')=[p(\mu)](\mu'),\end{equation}
for $\mu$ and $\mu'$ in ${\operatorname{Sym}}_{n-1}^{\mathbb Z} U$. 
\item\label{8.19.b} 
Define the element  $\pmb \delta$ of ${\mathfrak R}$ by  
$$\pmb \delta= [{\textstyle (\bigwedge^\mathrm{top} p)}(\Theta)](\Theta)\in {\mathfrak R}_{(0,\mathrm{top})}.$$ 
(It is reasonable to call $\pmb \delta$  ``the determinant''  of $p$.) 

\item\label{8.19.c}  Define the $\mathbb Z$-module homomorphism $q: D_{n-1}^{\mathbb Z}U^* \to {\mathfrak R}\otimes_{\mathbb Z}{\operatorname{Sym}}_{n-1}^{{\mathbb Z}}U$ by
$$\textstyle q(\nu)= [(\bigwedge^{\mathrm{top}-1}_{{\mathfrak R}}p)(\nu(\Theta))](\Theta) $$
for $\nu\in D_{n-1}U^*$. (It is reasonable to call $q$  ``the classical adjoint''  of $p$.) 

\item\label{manuf-data-d} Define the $\mathbb Z$-module homomorphism $\mathfrak Q: D_{n-1}^{\mathbb Z}U^* \otimes_{\mathbb Z} D_{n-1}^{\mathbb Z}U^*
\to {\mathfrak R}$ by $$\mathfrak Q(\nu\otimes \nu')= [q(\nu)](\nu'),$$ for $\nu$ and $\nu'$ in $D_{n-1}^{\mathbb Z}U^*$.

\end{enumerate}
\end{data}

\begin{remark}\label{square} The basis element $\Theta\otimes\Theta$ of the rank one free $\mathbb Z$-module $${\textstyle (\bigwedge^\mathrm{top}_{\mathbb Z} {\operatorname{Sym}}_{n-1}^{\mathbb Z}U)\otimes_{\mathbb Z} (\bigwedge^\mathrm{top} {\operatorname{Sym}}_{n-1}^{\mathbb Z}U)}$$ is uniquely determined because every unit in $\mathbb Z$ squares to one. This basis element appears in  (\ref{8.19.b}) and (\ref{8.19.c}) of Data \ref{manuf-data}. We conclude that Data \ref{manuf-data} has been described in a completely coordinate-free manner.
\end{remark}

We record a  list of obvious, but very useful, statements about the data of \ref{manuf-data}.
\begin{observation}\label{obvious}Adopt the data of {\rm \ref{Opening-Data}} and {\rm\ref{manuf-data}}. The following statements hold.
\begin{enumerate}[\rm(a)]

\item\label{obvious-a} If $\mu$ is in ${\operatorname{Sym}}_{n-1}^{{\mathbb Z}}U$, then $p(\mu)=\mu(\Phi)$ in ${\mathfrak R}\otimes_{{\mathbb Z}}D_{n-1}^{{\mathbb Z}}(U^*)$.

\item\label{obvious-aa} If $\mu$ and $\mu'$ are in ${\operatorname{Sym}}_{n-1}^{\mathbb Z}U$, then $[p(\mu)](\mu')=[p(\mu')](\mu)$ in  ${\mathfrak R}$.

\item\label{obvious-b} If $\mu$ is in ${\operatorname{Sym}}_{n-1}^{\mathbb Z}U$, then $q(p(\mu))=\pmb \delta\otimes \mu$ 
in ${\mathfrak R}\otimes_{{\mathbb Z}}{\operatorname{Sym}}_{n-1}^{\mathbb Z}U$.

\item\label{obvious-c} If $\nu$ is in $D_{n-1}^{\mathbb Z}U^*$, then 
$p(q(\nu))=\pmb \delta\otimes \nu$ in ${\mathfrak R}\otimes_{{\mathbb Z}}D_{n-1}^{{\mathbb Z}}(U^*)$.

\item\label{obvious-e} If $\mu$ is in ${\operatorname{Sym}}_{n-1}^{\mathbb Z}U$ and $\nu$ is in $D_{n-1}^{\mathbb Z}U^*$,
then 
$\mathfrak Q(p(\mu)\otimes \nu)=\pmb \delta \cdot \mu(\nu)$  in ${\mathfrak R}$.

\item\label{obvious-d} If $\nu$ is in $D_{n-1}^{\mathbb Z}U^*$ and $\mu$ is in ${\operatorname{Sym}}_{n-1}^{\mathbb Z}U$, then $\mathfrak Q(\nu\otimes p(\mu))=\pmb \delta\cdot \nu(\mu)$  in ${\mathfrak R}$.

\item\label{obvious-f} If $\nu$ and $\nu'$ are in $D_{n-1}^{\mathbb Z}(U^*)$, then 
 $\mathfrak Q(\nu\otimes \nu')=\mathfrak Q(\nu'\otimes \nu)$ in ${\mathfrak R}$.

\item\label{obvious-h} If $\nu$ and $\nu'$ are in $D_{n-1}^{\mathbb Z}(U^*)$, then $[q(\nu)](\nu')=[q(\nu')](\nu)$  in ${\mathfrak R}$.
\end{enumerate} 
\end{observation}

\begin{proof} {\bf(\ref{obvious-a}).} In light of Remark~\ref{Phi}, (\ref{flip-flop}), and (\ref{mod-act}),  $p(\mu)$ and $\mu(\Phi)$ both represent the ${\mathfrak R}$-module homomorphism ${\mathfrak R}\otimes_{{\mathbb Z}}{\operatorname{Sym}}_{n-1}^{{\mathbb Z}}U\to {\mathfrak R}$ which sends $\mu'$ in ${\operatorname{Sym}}_{n-1}^{{\mathbb Z}}U$ to
$$[p(\mu)](\mu')=\Phi(\mu\mu').$$

\smallskip\noindent {\bf(\ref{obvious-aa}).} The assertion    holds because  multiplication in ${\operatorname{Sym}}_{\bullet}^{{\mathbb Z}}U$ is commutative.

\smallskip\noindent {\bf(\ref{obvious-b}).} Observe  that $$\textstyle q(p(\mu))=[(\bigwedge^{\mathrm{top}-1}p)(p(\mu)(\Theta))](\Theta)
= [\mu[(\bigwedge^{\mathrm{top}}p)(\Theta)]](\Theta)=[(\bigwedge^{\mathrm{top}}p)(\Theta)](\Theta)\cdot \mu=\pmb\delta\otimes \mu.$$The second equality used (\ref{obvious-aa}).

\smallskip\noindent {\bf(\ref{obvious-c}).} Observe that $$\textstyle p(q(\nu))=p([(\bigwedge^{\mathrm{top}-1}p)(\nu(\Theta))](\Theta))=
[\nu(\Theta)][(\bigwedge^{\mathrm{top}}p)(\Theta)]= [\Theta[(\bigwedge^{\mathrm{top}}p)(\Theta)]]\cdot \nu
=\pmb\delta\otimes \nu.$$ Again, the second equality used (\ref{obvious-aa}). 
 
\smallskip\noindent {\bf(\ref{obvious-e}).} Observe  that $\textstyle \mathfrak Q(p(\mu)\otimes \nu)=[q(p(\mu))](\nu)=\pmb \delta\cdot \mu(\nu)$
by (\ref{obvious-b}).

\smallskip\noindent {\bf(\ref{obvious-f}).} The interplay between 
the module action of the ring $\bigwedge^{\bullet}_{{\mathbb Z}}({\operatorname{Sym}}_{n-1}^{{\mathbb Z}}U) $ on module $\bigwedge^{\bullet}_{{\mathbb Z}}(D_{n-1}^{{\mathbb Z}}(U^*)) $ and the module action of  the ring $\bigwedge^{\bullet}_{{\mathbb Z}}(D_{n-1}^{{\mathbb Z}}(U^*)) $ on the module $\bigwedge^{\bullet}_{{\mathbb Z}}({\operatorname{Sym}}_{n-1}^{{\mathbb Z}}U) $ yields   $$\textstyle \begin{array}{llll}\mathfrak Q(\nu\otimes \nu')&=&[q(\nu)](\nu')\vspace{5pt}\\&=&
\left[[(\bigwedge^{\mathrm{top}-1}p)(\nu(\Theta))](\Theta)\vphantom{{\widetilde{\Phi}}}\right] (\nu')\vspace{5pt}\\
&=&
\nu'\left[[(\bigwedge^{\mathrm{top}-1}p)(\nu(\Theta))](\Theta)\vphantom{{\widetilde{\Phi}}}\right] 
\vspace{5pt}\\&=&
\left[\nu'\wedge [(\bigwedge^{\mathrm{top}-1}p)(\nu(\Theta))]\vphantom{{\widetilde{\Phi}}}\right](\Theta)
\vspace{5pt}\\&=&
(-1)^{\mathrm{top}-1} \left[[(\bigwedge^{\mathrm{top}-1}p)(\nu(\Theta))]\wedge \nu'\vphantom{{\widetilde{\Phi}}}\right](\Theta) 
\vspace{5pt}\\&=&
(-1)^{\mathrm{top}-1} \left[(\bigwedge^{\mathrm{top}-1}p)(\nu(\Theta))\vphantom{{\widetilde{\Phi}}}\right] [\nu'(\Theta)]\vspace{5pt}\\ &=&
(-1)^{\mathrm{top}-1} \left[(\bigwedge^{\mathrm{top}-1}p)(\nu'(\Theta))\vphantom{{\widetilde{\Phi}}}\right] [\nu(\Theta)] & \text{ by (\ref{obvious-aa})}
\vspace{5pt}\\&=&
(-1)^{\mathrm{top}-1} \left[(\bigwedge^{\mathrm{top}-1}p)(\nu'(\Theta))\wedge \nu\vphantom{{\widetilde{\Phi}}}\right](\Theta)
\vspace{5pt}\\&=&
 \left[\nu\wedge (\bigwedge^{\mathrm{top}-1}p)(\nu'(\Theta))\vphantom{{\widetilde{\Phi}}}\right](\Theta)
\vspace{5pt}\\&=&
\nu \left[[ (\bigwedge^{\mathrm{top}-1}p)(\nu'(\Theta))](\Theta)\vphantom{{\widetilde{\Phi}}}\right]
\vspace{5pt}\\&=&
 \left[[ (\bigwedge^{\mathrm{top}-1}p)(\nu'(\Theta))](\Theta)\vphantom{{\widetilde{\Phi}}}\right](\nu)\vspace{5pt}\\
&=&[q(\nu')](\nu)=\mathfrak Q(\nu'\otimes \nu).\end{array}$$

\smallskip \noindent Assertions (\ref{obvious-d}) and (\ref{obvious-h}) now follows from (\ref{obvious-e}) and 
(\ref{obvious-f}).
\end{proof}

In our description of $(\mathbb B,b)$ in Definition \ref{main-def}, it is not necessary to name a complete basis for $U$; but our description does make use of a distinguished minimal generator ``$x$'' of $U$. (In other words, if the free $\mathbb Z$-module $U$ has basis $x,y,z$, then the maps and modules of $\mathbb B$ treat the basis vector $x$ differently than they treat $y$ and $z$; but  $\mathbb B$ is symmetric in $y$ and $z$.) The second collection of data which is manufactured from the basic data of \ref{Opening-Data} makes use of the distinguished element $x$.

\begin{data}\label{data-3} Adopt the data of \ref{Opening-Data}. Decompose the rank $3$ free $\mathbb Z$-module $U$ as $\mathbb Z x \oplus U_0$ for some element $x$ of $U$ and some rank $2$ free submodule $U_0$ of $U$. Let ${\widetilde{\Phi}}$ be the element of $D_{2n-1}^{\mathbb Z}(U^*)$ with $x({\widetilde{\Phi}})=\Phi$ in $D_{2n-2}^{\mathbb Z}(U^*)$ and $\mu({\widetilde{\Phi}})=0$ for all $\mu\in {\operatorname{Sym}}_{2n-1}^{\mathbb Z}(U_0)$.
\end{data}

An explicit version of ${\widetilde{\Phi}}$ may be found in (\ref{superPhi}). 

\smallskip It is possible, and not difficult, to phrase  Definition \ref{main-def} using a basis for $\bigwedge^2_{\mathbb Z} U_0$ instead of an   explicit basis $y,z$ for $U_0$. On the other hand, we have chosen to use the explicit basis $y,z$. The interested reader can easily re-write \ref{main-def} in terms of a basis for $\bigwedge^2_{\mathbb Z} U_0$.
\begin{definition}\label{main-def} Adopt the data of {\rm\ref{Opening-Data}, \ref{manuf-data}}, and {\rm\ref{data-3}}. Let $y, z$ be a basis for $U_0$.
\begin{enumerate}[\rm (a)]
\item\label{main-def-a} Define the $\mathbb Z$-module homomorphisms
$$\begin{array}{rlll}
\beta_1:&\bigwedge^2_{{\mathbb Z}}( D_{n}^{\mathbb Z}(U_0^*)) &\to& {\mathfrak R},\vspace{5pt}\\
\beta_2:&{\operatorname{Sym}}_{n-1}^{\mathbb Z}U_0\otimes_{{\mathbb Z}} D_{n}^{\mathbb Z}(U_0^*) &\to& {\mathfrak R}, \text{ and}\vspace{5pt}\\
\beta_3:& \bigwedge^2_{{\mathbb Z}}( {\operatorname{Sym}}_{n-1}^{\mathbb Z}U_0) &\to& {\mathfrak R}\end{array}$$
 by 
$$\begin{array}{lll}
\beta_1(\nu\wedge \nu')&=&\Psi(x)\cdot [\mathfrak Q(z(\nu)\otimes y(\nu'))- \mathfrak Q(y(\nu)\otimes z(\nu'))]\vspace{5pt}\\
\beta_2(\mu\otimes \nu)&=&\Psi(x)\cdot [\mathfrak Q((z\mu)({\widetilde{\Phi}})\otimes y(\nu))-\mathfrak Q((y\mu)({\widetilde{\Phi}})\otimes z(\nu))]\vspace{5pt}\\
\beta_3(\mu\wedge \mu')&=& \Psi(x)\cdot [\mathfrak Q((z\mu)({\widetilde{\Phi}})\otimes (y\mu')({\widetilde{\Phi}}))-\mathfrak Q((y\mu)({\widetilde{\Phi}})\otimes (z\mu')({\widetilde{\Phi}}))],\end{array}$$
for $\mu$ and $\mu'$ in ${\operatorname{Sym}}_{n-1}^{\mathbb Z}U_0$ and $\nu$ and $\nu'$ in $D_{n}^{\mathbb Z}(U_0^*)$.
\item\label{main-def-b} Define the free ${\mathfrak R}$-module $B_2$ and the ${\mathfrak R}$-module homomorphisms 
$$\mathfrak b: {\textstyle \bigwedge^2_{{\mathfrak R}}} B_2\to {\mathfrak R}\quad\text{and}\quad b_2: B_2\to B_2^*$$by
$$\begin{array}{rcl}B_2&=&  {\mathfrak R}\otimes_{\mathbb Z} ({\operatorname{Sym}}_{n-1}^{\mathbb Z}U_0\oplus D_n^{\mathbb Z}(U_0^*)),\vspace{5pt}\\
\mathfrak b((\mu+\nu)\wedge (\mu'+\nu'))&=&\begin{cases}\phantom{+}\beta_3(\mu\wedge \mu')+\beta_2(\mu\otimes\nu')-\beta_2(\mu'\otimes\nu)+\beta_1(\nu\wedge \nu')\vspace{5pt}\\+\Psi(y)\cdot \pmb \delta\cdot ( [z\mu](\nu')-[z\mu'](\nu))-\Psi(z)\cdot \pmb \delta\cdot ([y\mu](\nu')-[y\mu'](\nu)),\end{cases}\vspace{5pt}\\
&&\text{for $\mu,\mu'$ in ${\operatorname{Sym}}_{n-1}^{\mathbb Z}U_0$, $\nu,\nu'$ in $D_n^{\mathbb Z}(U_0^*)$, and}\vspace{5pt}
\\ 

 [b_2(\theta_2)](\theta_2')&=&\mathfrak b(\theta_2\wedge \theta_2'),\quad \text{for $\theta_2$ and $\theta_2'$ in $B_2$.} \end{array}$$ 

\item \label{main-def-c} Define $\mathbb B$ to be the  
sequence of free ${\mathfrak R}$-modules and ${\mathfrak R}$-module homomorphisms:

$$(\mathbb B,b):\quad 0\to B_3\xrightarrow{\ b_3\ } B_2 \xrightarrow{\ b_2\ } B_1\xrightarrow{\ b_1\ } B_0,$$
with $B_3=B_0={\mathfrak R}$, $B_2$ equal to the module described in (b),  $B_1=B_2^*$,

$$\begin{array}{ll}b_1(\nu)=\Psi(x)\cdot \Psi(q(\nu)),&\text{for $\nu\in D_{n-1}^{\mathbb Z}(U_0^*)$},\vspace{5pt}\\  
b_1(\mu)=\pmb \delta\cdot  \Psi(\mu)-\Psi(x)\cdot \Psi(q(\mu({\widetilde{\Phi}}))),&\text{for $\mu\in {\operatorname{Sym}}_n^{\mathbb Z}U_0$},\end{array}$$  $b_2$ equal to the homomorphism  described in (b), and $b_3(1)=b_1$.
\end{enumerate}\end{definition}

\begin{remarks}${}$\begin{enumerate}[(a)]\item Notice that $B_1=   {\mathfrak R}\otimes_{\mathbb Z} (D_{n-1}^{\mathbb Z}(U_0^*)\oplus {\operatorname{Sym}}_n^{\mathbb Z}U_0)$. \item We have written ``$\cdot$'' to emphasize that the multiplication is ordinary multiplication in the polynomial ring ${\mathfrak R}$.
\item An alternate description of the homomorphism $b_2$ is given in Observation~\ref{b2}.
\item A more explicit version of $(\mathbb B,b)$ is given in Proposition~\ref{explc}. 
\item Examples of specializations of $\mathbb B$ are given in Section~\ref{ex}.
\item The bi-homogeneous form of $(\mathbb B, b)$ is given just before Remark~\ref{10.1}.\end{enumerate}\end{remarks}

\section{The description of ${\mathbb G}$ as given in \cite{EKK}.}\label{EKK}

The following data has been taken from \cite{EKK}, especially Section 6. We employ  Conventions~\ref{conventions}.

\begin{data}\label{Mar-5-2014} Adopt  Data \ref{Opening-Data}. 
\begin{enumerate}[(a)]\item Let $L_{a,b}$ and $K_{a,b}$ represent the free $\mathbb Z$-modules $L_{a,b}^{\mathbb Z}U$
and $K_{a,b}^{\mathbb Z}U$, respectively, as described in \cite[Data~2.1]{EKK}. 
In particular, 
$$\textstyle \begin{array}{lll}
L_{a,b}&=&\ker \left(\bigwedge^a_{\mathbb Z}U\otimes_{{\mathbb Z}}{\operatorname{Sym}}_b^{{\mathbb Z}}U\xrightarrow{\kappa}
\bigwedge^{a-1}_{\mathbb Z}U\otimes_{{\mathbb Z}}{\operatorname{Sym}}_{b+1}^{{\mathbb Z}}U\right)\quad\text{and}\vspace{5pt}\\
K_{a,b}&=&\ker \left(\bigwedge^a_{\mathbb Z}U\otimes_{{\mathbb Z}}D_b^{{\mathbb Z}}(U^*)\xrightarrow{\eta}
\bigwedge^{a-1}_{\mathbb Z}U\otimes_{{\mathbb Z}}D_{b-1}^{{\mathbb Z}}(U^*)\right),\end{array}$$where $\kappa$ is a Koszul complex map and $\eta$ is an Eagon-Northcott complex map.
\item Consider the map of complexes

\begin{equation}\label{MOC}\xymatrix{& {\mathfrak R}\otimes_{{\mathbb Z}}L_{2,n}\ar[r]^{h_3}\ar[d]^{\underbar v_3}&{\mathfrak R}\otimes_{{\mathbb Z}}L_{1,n}\ar[r]^{h_2}\ar[d]^{\underbar v_2}&{\mathfrak R}\otimes_{{\mathbb Z}}L_{0,n}\ar[r]^{h_1}\ar[d]^{\underbar v_1}&{\mathfrak R}\vspace{5pt}\\
{\mathfrak R}\otimes_{{\mathbb Z}}\bigwedge^3_{\mathbb Z}U\ar[r]^{h_3'}&{\mathfrak R}\otimes_{{\mathbb Z}}K_{2,n-2}\ar[r]^{h_2'}&{\mathfrak R}\otimes_{{\mathbb Z}}K_{1,n-2}\ar[r]^{h_1'}&{\mathfrak R}\otimes_{{\mathbb Z}}K_{0,n-2},}\end{equation}
of \cite[Obv.~4.2]{EKK}. The vertical map $\underbar v_i:{\mathfrak R}\otimes_{{\mathbb Z}}L_{i-1,n} \to {\mathfrak R}\otimes_{{\mathbb Z}}K_{i-1,n-2}$
is induced by \begin{equation}\label{v1}{\operatorname{Sym}}_{n}^{\mathbb Z}U\to {\mathfrak R}\otimes_{\mathbb Z}D_{n-2}(U^*)\quad\text{with}\quad \mu\mapsto \mu(\Phi)\text{ for $\mu\in {\operatorname{Sym}}_n^{{\mathbb Z}}U$};\end{equation}the horizontal map $h_1$ is induced by $\Psi:{\operatorname{Sym}}_n^{{\mathbb Z}}U\to {\mathfrak R}$; the horizontal map $$h_i:{\mathfrak R}\otimes_{{\mathbb Z}}L_{i-1,n}\to {\mathfrak R}\otimes_{{\mathbb Z}}L_{i-2,n},\quad\text{for $2\le i\le 3$},$$
is induced by the Koszul complex map $\bigwedge^{i-1}_{{\mathbb Z}}U\to {\mathfrak R}\otimes_{{\mathbb Z}} \bigwedge^{i-2}_{{\mathbb Z}}U$ with \begin{equation}\label{KCM}u_1\wedge\cdots\wedge u_{i-1}\mapsto \sum\limits_{j=1}^{i-1}(-1)^{j+1}\Psi(u_j)\otimes u_1\wedge\cdots\wedge\widehat{u_j}\wedge\cdots\wedge u_{i-1} \text{ for $u_j\in U$;}\end{equation}the horizontal map $h_i':{\mathfrak R}\otimes_{{\mathbb Z}}K_{i,n-2}\to {\mathfrak R}\otimes_{{\mathbb Z}}K_{i-1,n-2}$, for $1\le i\le 2$, is induced by the Koszul complex map $\bigwedge^{i}_{{\mathbb Z}}U\to {\mathfrak R}\otimes_{{\mathbb Z}} \bigwedge^{i-1}_{{\mathbb Z}}U$
which is analogous to (\ref{KCM}); and 
the horizontal map $$\textstyle h_3':{\mathfrak R}\otimes_{{\mathbb Z}}\bigwedge^3_{{\mathbb Z}}U\to {\mathfrak R}\otimes_{{\mathbb Z}}K_{2,n-2}$$ is induced by $$\textstyle \bigwedge^3_{{\mathbb Z}}U\to {\mathfrak R}\otimes_{{\mathbb Z}} \bigwedge^2_{{\mathbb Z}}U\otimes_{{\mathbb Z}}D_{n-2}^{\mathbb Z}(U^*)$$ with
$$u_1\wedge u_2\wedge u_3\mapsto \sum\limits_\ell \Psi(m_\ell)\otimes [u_1\wedge u_2\otimes u_3(m_\ell^*) 
-u_1\wedge u_3\otimes u_2(m_\ell^*)
+u_2\wedge u_3\otimes u_1(m_\ell^*)],$$ where $(\{m_\ell\},\{m_\ell^*\})$ is any pair of dual bases for ${\operatorname{Sym}}_{n-1}^{{\mathbb Z}}U$ and $D_{n-1}^{{\mathbb Z}}(U^*)$, respectively, and the $u_j$ are elements of $U$.

\item\label{Mar-5-2014-c} Follow the lead of \cite[Def.~6.6 and Thm.~6.15]{EKK} and consider the complex of ${\mathfrak R}$-module homomorphisms
\begin{equation}\label{Bbb-G}(\mathbb G,g):\quad 0\to G_3\xrightarrow{\ g_3\ }G_2\xrightarrow{\ g_2\ }G_1\xrightarrow{\ g_1\ }{\mathfrak R},\end{equation}with $G_i=\ker \underbar v_i$, for $1\le i\le 2$, and $G_3={\mathfrak R}\otimes_{{\mathbb Z}}\bigwedge_{{\mathbb Z}}^3U$. (In \cite{EKK} this complex is called $\widetilde{\mathbb G}'(n)$.) The ${\mathfrak R}$-module homomorphism $g_1$ is induced by $G_1\subseteq {\mathfrak R}\otimes_{{\mathbb Z}}{\operatorname{Sym}}_nU\xrightarrow{\Psi}{\mathfrak R}$; the ${\mathfrak R}$-module homomorphism $g_2$ is induced by 
$$\textstyle G_2\subseteq {\mathfrak R}\otimes_{{\mathbb Z}}\bigwedge_{{\mathbb Z}}^1 U\otimes_{{\mathbb Z}}{\operatorname{Sym}}_n^{\mathbb Z}U \xrightarrow{\ 1\otimes \Psi\otimes 1\ }{\mathfrak R}\otimes_{{\mathbb Z}}{\operatorname{Sym}}_n^{{\mathbb Z}}U\supset G_1;$$ and the 
${\mathfrak R}$-module homomorphism $g_3$ is induced by 

$$\begin{array}{lcl}
\textstyle G_3={\mathfrak R}\otimes_{{\mathbb Z}}\bigwedge_{{\mathbb Z}}^3U&\xrightarrow{\ q\circ \operatorname{ev}^*\ }&
{\mathfrak R}\otimes_{{\mathbb Z}}\bigwedge_{{\mathbb Z}}^3U\otimes_{{\mathbb Z}}{\operatorname{Sym}}_{n-1}^{{\mathbb Z}} U\vspace{5pt}\\&\xrightarrow{\ \operatorname{Kos}^{\Psi}\circ \kappa\ }&
{\mathfrak R}\otimes_{{\mathbb Z}}\bigwedge_{{\mathbb Z}}^1U\otimes_{{\mathbb Z}}{\operatorname{Sym}}_{n}^{{\mathbb Z}} U\supset {\mathfrak R}\otimes_{{\mathbb Z}} L_{1,n}\supset G_2.\end{array}$$
The ${\mathfrak R}$-module homomorphism $q\circ \operatorname{ev}^*$ sends the element $\omega\in \bigwedge_{{\mathbb Z}}^3U$,  to
$$\textstyle (q\circ \operatorname{ev}^*)(\omega)=
\sum\limits_\ell \Psi(m_\ell)\otimes \omega \otimes q(m_\ell^*)\in {\mathfrak R}\otimes_{{\mathbb Z}}\bigwedge_{{\mathbb Z}}^3U\otimes_{{\mathbb Z}}{\operatorname{Sym}}_{n-1}^{{\mathbb Z}} U,$$
  where $(\{m_\ell\},\{m_\ell^*\})$ is any pair of dual bases for ${\operatorname{Sym}}_{n-1}^{{\mathbb Z}}U$ and $D_{n-1}^{{\mathbb Z}}(U^*)$.  The ${\mathfrak R}$-module homomorphism $\operatorname{Kos}^{\Psi}\circ \kappa$ sends the element $1\otimes (u_1\wedge u_2\wedge u_3)\otimes \mu$ of ${\mathfrak R}\otimes_{{\mathbb Z}}\bigwedge_{{\mathbb Z}}^3U\otimes_{{\mathbb Z}}{\operatorname{Sym}}_{n-1}^{{\mathbb Z}} U$, with $u_j\in U$ and $\mu\in {\operatorname{Sym}}_{n-1}^{{\mathbb Z}}U$,
to 
$$\begin{array}{rl}&(\operatorname{Kos}^{\Psi}\circ \kappa)(1\otimes (u_1\wedge u_2\wedge u_3)\otimes \mu)\vspace{5pt}\\=&\begin{cases}
\phantom{+}\Psi(u_1)\otimes u_2\otimes u_3\mu
-\Psi(u_2)\otimes u_1\otimes u_3\mu
-\Psi(u_1)\otimes u_3\otimes u_2\mu\vspace{5pt}\\
+\Psi(u_3)\otimes u_1\otimes u_2\mu
+\Psi(u_2)\otimes u_3\otimes u_1\mu
-\Psi(u_3)\otimes u_2\otimes u_1\mu\end{cases}\end{array}$$
in ${\mathfrak R}\otimes_{{\mathbb Z}}\bigwedge_{{\mathbb Z}}^1U\otimes_{{\mathbb Z}}{\operatorname{Sym}}_{n}^{{\mathbb Z}} U$.

\item As was observed in Remark \ref{Phi}, $\Phi$ is naturally an element of the 
$${\mathfrak R}\otimes_{{\mathbb Z}}D_{2n-2}^{{\mathbb Z}}(U^*)=D_{2n-2}^{{\mathfrak R}}({\mathfrak R}\otimes_{{\mathbb Z}}U^*).$$
Let $I=\operatorname{ann}(\Phi)$. In other words,
$$I=\{r\in {\mathfrak R}\mid r(\Phi)=0 \in {\mathfrak R}\otimes_{{\mathbb Z}}D_{2n-2}^{{\mathbb Z}}(U^*)\},$$ where $r(\Phi)$ represents the ${\mathfrak R}$-module action of $r$ on $\Phi$. 
\end{enumerate}
\end{data}

The following result is established in \cite[Thms.~6.15 and 4.16]{EKK}. For item (\ref{p-of-p}) one must also use 
the ``Persistence of Perfection Principle'', which is also known as  the ``transfer of perfection'' (see 
\cite[Prop.~6.14]{Ho-T} or \cite[Thm.~3.5]{BV}).

\begin{theorem}\label{old-mvd} Adopt Data~{\rm{\ref{Mar-5-2014}}}. The following statements hold.
\begin{enumerate}[\rm(1)]
\item The ${\mathfrak R}$-module homomorphisms $(\mathbb G,g)$ form a complex.
\item The localization $\mathbb G_{\pmb \delta}$ is a  resolution of ${\mathfrak R}_{\pmb \delta}/I{\mathfrak R}_{\pmb \delta}$ by projective ${\mathfrak R}_{\pmb\delta}$-modules. {\rm(}The projective ${\mathfrak R}_{\pmb\delta}$-modules  $(G_1)_{\pmb\delta}$ and $(G_2)_{\pmb \delta}$ both have rank $2n+1$. The module $G_3$ is isomorphic to ${\mathfrak R}$.{\rm)}
\item\label{p-of-p} If $S$ is a Noetherian ring, $\rho:{\mathfrak R}\to S$ is a ring homomorphism, $(\rho({\mathfrak R}_{(1,0)}))$ is an ideal of $S$ of grade at least $3$, and $\rho(\pmb \delta)$ is a unit of $S$, then $S\otimes_{{\mathfrak R}}\mathbb G$ is a resolution of $S/\rho(I)$ by projective $S$-modules.
\item Let $S$ be the standard-graded polynomial ring ${\pmb k}\otimes_{{\mathbb Z}}{\operatorname{Sym}}_\bullet^{{\mathbb Z}}U$, where ${\pmb k}$ is a field, and let $\rho: {\mathfrak R}\to S$ be a ${\operatorname{Sym}}_\bullet^{{\mathbb Z}}U$-algebra homomorphism, with 
\begin{enumerate}[\rm(i)]\item
$\rho({\mathfrak R}_{(0,1)})\in {\pmb k}$,   \item $\rho(\pmb \delta)$ is a unit in ${\pmb k}$, and 
\item $\Phi_{\rho}=\sum_{i}\rho(\Phi(m_i))m_i^*$ is the element in $D_{2n-2}^{\pmb k}(S_1^*)$ which corresponds to $\Phi$, where $\{m_i\}$, $\{m_i^*\}$ is any pair of dual bases for 
${\operatorname{Sym}}_{2n-2}^{\mathbb Z} U$ and $D_{2n-2}^{\mathbb Z}(U^*)$, respectively,

\end{enumerate}
then 
\begin{enumerate}[\rm(a)]
\item $S\otimes_{\mathfrak R} \mathbb G$ is a minimal homogeneous resolution of $S/\operatorname{ann} (\Phi_{\rho})$ by free $S$-modules 
\item $S\otimes_{\mathfrak R} \mathbb G$ is a Gorenstein-linear resolution of the form
\begin{equation}\label{form-mvd}0\to S(-2n-1)\to S(-n-1)^{2n+1}\to S(-n)^{2n+1}\to S\end{equation}
\item\label{by-way-of} every Gorenstein-linear resolution of the form {\rm{(\ref{form-mvd})}}  is obtained as $S\otimes_{\mathfrak R} \mathbb G$ for some ${\operatorname{Sym}}_\bullet^{{\mathbb Z}}U$-algebra homomorphism $\rho: {\mathfrak R}\to S$ which satisfies {\rm (i)\,--\,(iii)}.

\end{enumerate}\end{enumerate}
\end{theorem}
\begin{remark} The paper \cite{EKK} only promises that the ${\mathfrak R}_{\pmb \delta}$-modules $(G_1)_{\pmb\delta}$ and $(G_2)_{\pmb\delta}$ of Theorem~\ref{old-mvd} are projective. In Lemma~\ref{main-lemma}, we  prove that  the ${\mathfrak R}_{\pmb \delta}$-modules $(G_1)_{\pmb\delta}$ and $(G_2)_{\pmb\delta}$  are free. So, Lemma~\ref{main-lemma} shows that each  ``projective'' in Theorem~\ref{old-mvd} may be replaced by ``free''.\end{remark}

\section{The main theorem}\label{main-theorem} The complex $(\mathbb G,g)$ of Theorem~\ref{old-mvd} has all of the desired properties, except, one is not able to answer the (very basic) questions, ``What \underline{exactly} are $G_1$, $G_2$, $g_1$, and $g_2$?'' The explicitly constructed complex $(\mathbb B,b)$ is our remedy to this defect of $(\mathbb G,g)$. 
The main result in the paper is Theorem~\ref{main}.

\begin{theorem}\label{main}Let $A$ be a standard-graded, Artinian, Gorenstein algebra over a field $\pmb k$. If the embedding codimension of $A$ is three and  
the minimal homogeneous resolution of $A$ by free $\operatorname{Sym}_{\bullet}^{\pmb k}A_1$-modules is Gorenstein-linear, then
$\operatorname{Sym}_{\bullet}^{\pmb k}A_1\otimes_{\mathfrak R}\mathbb B$ is
 a minimal homogeneous  resolution of 
$A$ by free $\operatorname{Sym}_{\bullet}^{\pmb k}A_1$-modules, where $(\mathbb B,b)$ is  
the  complex  of Definition~{\rm{\ref{main-def}.\ref{main-def-c}}}. 
Furthermore, $\operatorname{Sym}_{\bullet}^{\pmb k}A_1\otimes_{\mathfrak R}\mathbb B$ is explicitly constructed in a polynomial manner from the coefficients of the Macaulay inverse system for $A$. 
\end{theorem}

\noindent Theorem~\ref{main} is an immediate consequence of  Lemma~\ref{main-lemma}.\ref{main-lemma-g} by way of Theorem~{\rm\ref{old-mvd}.\ref{by-way-of}}.

Roughly speaking, in order to prove Lemma~\ref{main-lemma} (and hence Theorem~\ref{main}) one must
identify a nice generating set for $(G_1)_{\pmb \delta}$ and $(G_2)_{\pmb \delta}$ and one must write $g_1$ and $g_2$ in terms of this nice generating set. In fact, we reformulate $g_1$ and $g_2$ first, in Lemma~\ref{comm-diag}, and then we reformulate  $(G_1)_{\pmb \delta}$ and $(G_2)_{\pmb \delta}$ in Lemma~\ref{main-lemma} (\ref{main-lemma-a}) and (\ref{main-lemma-b}). The proof of Lemma~\ref{comm-diag} is not particularly hard; but it is long. On the other hand, Lemma~\ref{comm-diag} is the central calculation in the paper.

 Ultimately, in Lemma~\ref{main-lemma}.\ref{main-lemma-f}, we produce an isomorphism of complexes $\tau:(\mathbb B,b)\to (\mathbb E,e)$, where $(\mathbb E,e)$ is a sub-complex of $(\mathbb G,g)$ with 
$(\mathbb E,e)_{\pmb \delta}=(\mathbb G,g)_{\pmb \delta}$. We begin by defining the critical homomorphisms $\tau_i$ in Definition~\ref{tau} and Observation~\ref{really}.
 In Lemma~\ref{comm-diag} 
 we show that the homomorphisms $\tau:(\mathbb B,b)\to (\mathbb G,g)$ form a commutative diagram.

\begin{definition}\label{tau} Recall the sequence of ${\mathfrak R}$-module homomorphisms $(\mathbb B,b)$ of Definition {\rm\ref{main-def}.\ref{main-def-c}} and the complex of ${\mathfrak R}$-modules $(\mathbb G,g)$ of Data {\rm \ref{Mar-5-2014}.\ref{Mar-5-2014-c}}. Define ${\mathfrak R}$-module homomorphisms $\tau_i:B_i\to G_i$ as follows.
\begin{enumerate}[\rm(a)]
\item Let $\tau_0:B_0={\mathfrak R}\to {\mathfrak R}=G_0$ be the identity map.
\item Let $\tau_1:B_1\to G_1$ be the
 ${\mathfrak R}$-module homomorphism defined 
by $$\tau_1 (\nu+\mu)=xq(\nu)+\pmb\delta \mu-xq(\mu({\widetilde{\Phi}}))\in {\mathfrak R}\otimes_{{\mathbb Z}}{\operatorname{Sym}}_n^{{\mathbb Z}}U,$$ for $\nu\in D_{n-1}^{{\mathbb Z}}(U_0^*)$ and $\mu\in {\operatorname{Sym}}_n^{{\mathbb Z}}U_0$. (We show in Observation~\ref{really} that the image of $\tau_1$ is contained in $G_1$.)

\item Let $\tau_2:B_2\to G_2$ be the
 ${\mathfrak R}$-module homomorphism defined 
by $$\tau_2 (\mu+\nu)= \begin{cases}
 \phantom{+}\pmb\delta\kappa(x\wedge z\otimes q((y\mu)({\widetilde{\Phi}})))-\pmb\delta\kappa (x \wedge y \otimes q((z\mu)({\widetilde{\Phi}})))-\pmb \delta^2 \kappa (y\wedge z\otimes \mu)\vspace{5pt}\\
+
  \pmb\delta\kappa (x\wedge z \otimes q(y(\nu))) -\pmb\delta\kappa (x\wedge y \otimes q(z(\nu))),&\text{in }{\mathfrak R}\otimes_{Z}L_{1,n}
,\end{cases}$$ for $\mu\in {\operatorname{Sym}}_{n-1}^{{\mathbb Z}}U_0$  and  $\nu\in D_{n}^{{\mathbb Z}}(U_0^*)$.
(We show in Observation~\ref{really} that the image of $\tau_2$ is contained in $G_2$.)
\item Let $\tau_3:B_3={\mathfrak R} \to {\mathfrak R}\otimes_{{\mathbb Z}}\bigwedge_{{\mathbb Z}}^3U=G_3$ be the
 ${\mathfrak R}$-module homomorphism defined $\tau_3(1)=\pmb \delta^2 x\wedge y\wedge z$.

\end{enumerate}
\end{definition}

\begin{observation}\label{really} The image of each homomorphism $\tau_i$, from Definition~{\rm\ref{tau}}, is in $G_i$.  \end{observation}
\begin{proof}We need only discuss $\tau_1$ and $\tau_2$. The module $G_1$ 
is defined to be the kernel of the ${\mathfrak R}$-module homomorphism 
$$\underbar v_1:  {\mathfrak R} \otimes_{{\mathbb Z}}L_{0,n}={\mathfrak R} \otimes_{{\mathbb Z}}{\operatorname{Sym}}_{n}^{{\mathbb Z}}U\to {\mathfrak R} \otimes_{{\mathbb Z}}D_{n-2}^{{\mathbb Z}}(U^*).$$ We  verify that the image of $\tau_1$ is contained in the kernel of $\underbar v_1$.
If $\nu\in D_{n-1}U_0^*$, then
$$\begin{array}{llll}\underbar v_1(\tau_1(\nu))=\underbar v_1(xq(\nu))&=&[xq(\nu)](\Phi)&{\rm(\ref{v1})}\vspace{5pt}\\&=&x\left( [q(\nu)](\Phi)\right)&{\rm(\ref{mod-act})}\vspace{5pt}\\&=&x(p(q(\nu))&{\rm(\ref{obvious}.\ref{obvious-a})}\vspace{5pt}\\&=&\pmb \delta x(\nu)
&{\rm(\ref{obvious}.\ref{obvious-c})}
\vspace{5pt}\\&=&0&{\rm(\ref{xim})},\end{array}$$ and if 
$\mu$ in ${\operatorname{Sym}}_n^{{\mathbb Z}}U_0$, then
\begin{longtable}{llll}$\underbar v_1(\tau_1(\mu))=\underbar v_1\left(\pmb \delta \mu- xq(\mu({\widetilde{\Phi}}))\right)$
&$=$&$\pmb \delta \mu(\Phi)- [xq(\mu({\widetilde{\Phi}}))](\Phi)$&{\rm(\ref{v1})}\vspace{5pt}\\
&$=$&$\pmb \delta \mu(\Phi)- x\left([q(\mu({\widetilde{\Phi}}))](\Phi)\right)$&{\rm(\ref{mod-act})}\vspace{5pt}\\
&$=$&$\pmb \delta \mu(\Phi)- x\left( p[q(\mu({\widetilde{\Phi}}))]\right)$&{\rm(\ref{obvious}.\ref{obvious-a})}\vspace{5pt}\\
&$=$&$\pmb \delta \mu(\Phi)- \pmb \delta x\left( \mu({\widetilde{\Phi}})\right)$&{\rm(\ref{obvious}.\ref{obvious-c})}\vspace{5pt}\\
&$=$&$\pmb \delta \mu(\Phi)- \pmb \delta \mu\left( x({\widetilde{\Phi}})\right)$&{\rm(\ref{mod-act})}\vspace{5pt}\\
&$=$&$\pmb \delta \mu(\Phi)- \pmb \delta \mu (\Phi)$&{\rm(\ref{data-3})}\vspace{5pt}\\
&$=$&$0$.
\end{longtable}
 The module $L_{1,n}$ is equal to the submodule $\kappa (\bigwedge^2_{{\mathbb Z}}U\otimes_{{\mathbb Z}}{\operatorname{Sym}}_{n-1}^{{\mathbb Z}}U)$
of $\bigwedge^1_{{\mathbb Z}}U\otimes_{{\mathbb Z}}{\operatorname{Sym}}_{n}^{{\mathbb Z}}U$;
 and therefore the image of $\tau_2$ is automatically contained in  ${\mathfrak R}\otimes_{{\mathbb Z}}L_{1,n}$. We still must verify that $\tau_2(B_2)$ is contained in $G_2$, which is defined to be the kernel of the ${\mathfrak R}$-module homomorphism
$$\underbar v_2:{\mathfrak R}\otimes_{{\mathbb Z}}L_{1,n}\to {\mathfrak R}\otimes_{{\mathbb Z}}K_{1,n-2}.$$
Fix $\mu\in {\operatorname{Sym}}_{n-1}^{{\mathbb Z}}U_0$. The definition of $\kappa$ yields
$$\textstyle \kappa (x\wedge y \otimes q((z\mu)({\widetilde{\Phi}})))=y\otimes xq((z\mu)({\widetilde{\Phi}}))-x\otimes yq((z\mu)({\widetilde{\Phi}}))\in {\mathfrak R}\otimes_{{\mathbb Z}}\bigwedge^1_{{\mathbb Z}}U\otimes _{{\mathbb Z}}{\operatorname{Sym}}_{n}^{{\mathbb Z}}U;$$so, calculations similar to the calculations of $\ker \underbar v_1$ yield that 
$$\begin{array}{lll}\underbar v_2\left(\kappa (x\wedge z \otimes q((y\mu)({\widetilde{\Phi}})))\right)
&=&z\otimes [xq((y\mu)({\widetilde{\Phi}}))](\Phi)-x\otimes [zq((y\mu)({\widetilde{\Phi}}))](\Phi)\vspace{5pt}\\
&=&z\otimes x\left([q((y\mu)({\widetilde{\Phi}}))](\Phi)\right)-x\otimes z\left([q((y\mu)({\widetilde{\Phi}}))](\Phi)\right)\vspace{5pt}\\
&=&\pmb\delta \left[\phantom{-}z\otimes x((y\mu)({\widetilde{\Phi}}))-x\otimes z((y\mu)({\widetilde{\Phi}}))\right]\vspace{5pt}\\
&=&\pmb\delta \left[\phantom{-}z\otimes (y\mu)(\Phi)-x\otimes (yz\mu)({\widetilde{\Phi}})\right],\vspace{5pt}\\
\underbar v_2\left(-\kappa (x\wedge y \otimes q((z\mu)({\widetilde{\Phi}})))\right)
&=&\pmb\delta \left[-y\otimes (z\mu)(\Phi)+x\otimes (yz\mu)({\widetilde{\Phi}})\right],\quad\text{and}\vspace{5pt}\\
\underbar v_2\left(-\pmb\delta\kappa (y\wedge z\otimes \mu))\right)
&=&\pmb\delta \left[-z\otimes (y\mu)(\Phi)+y\otimes (z\mu)(\Phi)\vphantom{{\widetilde{\Phi}}}\right]; 
\end{array}$$
hence, $\tau_2(\mu)$, which is equal to $$\pmb\delta\left(\kappa (x\wedge z \otimes q((y\mu)({\widetilde{\Phi}})))-\kappa (x\wedge y \otimes q((z\mu)({\widetilde{\Phi}})))-\pmb\delta\kappa (y\wedge z\otimes \mu)\right),$$ is in $\ker \underbar v_2=G_2$. Fix $\nu\in D_{n}^{{\mathbb Z}}(U_0^*)$. Observe that
$$\begin{array}{lll}\underbar v_2(\tau_2(\nu))&=&\pmb\delta\underbar v_2\left(\vphantom{\widetilde{\Phi}}
\kappa (x\wedge z \otimes q(y(\nu)))-
\kappa (x\wedge y \otimes q(z(\nu))) \right)\vspace{5pt}\\
&=&\pmb\delta^2[z\otimes xy(\nu)-x\otimes yz(\nu)-y\otimes xz(\nu)+x\otimes yz(\nu)]\end{array}$$ and this is zero because $x(\nu)=0$. It follows that $\tau_2(\nu)$ also is in $G_2$.
\end{proof}

We re-write the differential $b_2$ of    Definition~\ref{main-def} in a manner that is a little easier to use. This calculation makes use of the facts 
\begin{equation}\label{FT}\sum\limits_{m\in\binom {y,z}{r}} m^*(\mu)\cdot m = \mu 
\quad\text{ and }\quad \sum\limits_{m\in\binom {y,z}{r}} m(\nu)\cdot m^* = \nu\end{equation} 
 for all $\mu$ in ${\operatorname{Sym}}_r^{{\mathbb Z}}U_0$  and all $\nu\in D_r^{{\mathbb Z}}(U_0^*)$ for any non-negative integer $r$.
\begin{observation}\label{b2}
\begin{enumerate}[\rm (a)]
\item\label{b2-a} If $\mu\in {\operatorname{Sym}}_{n-1}^{{\mathbb Z}}U_0$, then 
$$b_2(\mu)=\begin{cases}
\phantom{+}
x\otimes \sum\limits_{m_1\in \binom{y,z}{n-1}}
\left[\mathfrak Q((z\mu)({\widetilde{\Phi}})\otimes (ym_1)({\widetilde{\Phi}}))-\mathfrak Q((y\mu)({\widetilde{\Phi}})\otimes (zm_1)({\widetilde{\Phi}}))\right]
\cdot m_1^*
\vspace{5pt}\\
+x\otimes \sum\limits_{m_1\in \binom{y,z}{n}}
\left[\vphantom{{\widetilde{\Phi}}}\mathfrak Q((z\mu)({\widetilde{\Phi}})\otimes y(m_1^*))-\mathfrak Q((y\mu)({\widetilde{\Phi}})\otimes z(m_1^*))\right ]\cdot m_1\vspace{5pt}\\
+
y\otimes \pmb \delta z\mu- z\otimes\pmb \delta y\mu.
\end{cases}$$
\item\label{b2-b} If $\nu\in D_{n}^{{\mathbb Z}}(U_0^*)$, then 
$$b_2(\nu)=\begin{cases}
-y\otimes\pmb\delta\cdot z(\nu)+z\otimes\pmb\delta\cdot y(\nu)
 \vspace{5pt}\\
-x\otimes \sum\limits_{m_1\in \binom{y,z}{n-1}} 
\left[\mathfrak Q((zm_1)({\widetilde{\Phi}})\otimes y(\nu))-\mathfrak Q((ym_1)({\widetilde{\Phi}})\otimes z(\nu))\right]\cdot m_1^* \vspace{5pt}\\
+x\otimes \sum\limits_{m_1\in \binom{y,z}{n}} 
\left[\vphantom{{\widetilde{\Phi}}} \mathfrak Q(z(\nu)\otimes y(m_1^*))-\mathfrak Q(y(\nu)\otimes z(m_1^*))
\right]
\cdot m_1.
\end{cases}$$
\end{enumerate}
\end{observation}
\begin{proof}{\bf(\ref{b2-a}).} If $\mu\in {\operatorname{Sym}}_{n-1}^{{\mathbb Z}}U_0$, then  
\begin{longtable}{lll}$b_2(\mu)$&$=$&$
\sum\limits_{m_1\in \binom{y,z}{n-1}}[b_2(\mu)](m_1)\cdot m_1^*+\sum\limits_{m_1\in \binom{y,z}{n}}[b_2(\mu)](m_1^*)\cdot m_1$\vspace{5pt}\\
&$=$&$\sum\limits_{m_1\in \binom{y,z}{n-1}}\mathfrak b(\mu\wedge m_1)\cdot m_1^*+\sum\limits_{m_1\in \binom{y,z}{n}}
\mathfrak b(\mu\wedge m_1^*)
\cdot m_1$\vspace{5pt}\\
&$=$&$\begin{cases} \phantom{+}\sum\limits_{m_1\in \binom{y,z}{n-1}}\beta_3(\mu\wedge m_1)
\cdot m_1^*\vspace{5pt}\\
+\sum\limits_{m_1\in \binom{y,z}{n}}
\left(\vphantom{{\widetilde{\Phi}}}\beta_2 (\mu\otimes m_1^*)+
\Psi(y)\cdot\pmb \delta [z\mu](m_1^*)-\Psi(z)\cdot\pmb \delta [y\mu](m_1^*)\right)
\cdot m_1
\end{cases}$\vspace{5pt}\\
&$=$&$\begin{cases}\phantom{+}
\sum\limits_{m_1\in \binom{y,z}{n-1}}
\Psi(x)\cdot [\mathfrak Q((z\mu)({\widetilde{\Phi}})\otimes (ym_1)({\widetilde{\Phi}}))-\mathfrak Q((y\mu)({\widetilde{\Phi}})\otimes (zm_1)({\widetilde{\Phi}}))]
\cdot m_1^*\vspace{5pt}\\
+\sum\limits_{m_1\in \binom{y,z}{n}}
\Psi(x)\cdot \left[\vphantom{{\widetilde{\Phi}}}\mathfrak Q((z\mu)({\widetilde{\Phi}})\otimes y(m_1^*))-\mathfrak Q((y\mu)({\widetilde{\Phi}})\otimes z(m_1^*))\right ]\cdot m_1
\vspace{5pt}\\+
\Psi(y)\cdot\pmb \delta z\mu-\Psi(z)\cdot\pmb \delta y\mu
\end{cases}$\\\vspace{5pt}
&$=$&$\begin{cases}
\phantom{+}
x\otimes \sum\limits_{m_1\in \binom{y,z}{n-1}}
\left[\mathfrak Q((z\mu)({\widetilde{\Phi}})\otimes (ym_1)({\widetilde{\Phi}}))-\mathfrak Q((y\mu)({\widetilde{\Phi}})\otimes (zm_1)({\widetilde{\Phi}}))\right]
\cdot m_1^*
\vspace{5pt}\\
+x\otimes \sum\limits_{m_1\in \binom{y,z}{n}}
\left[\vphantom{{\widetilde{\Phi}}}\mathfrak Q((z\mu)({\widetilde{\Phi}})\otimes y(m_1^*))-\mathfrak Q((y\mu)({\widetilde{\Phi}})\otimes z(m_1^*))\right ]\cdot m_1
\vspace{5pt}\\+
y\otimes \pmb \delta z\mu- z\otimes\pmb \delta y\mu.\end{cases}$
\end{longtable}

\medskip\noindent{\bf(\ref{b2-b}).} In a similar manner, if $\nu\in D_n^{{\mathbb Z}}(U_0^*)$, then 
\begin{longtable}{lll}

$ b_2(\nu)$&=&$\sum\limits_{m_1\in \binom{y,z}{n-1}} [b_2(\nu)](m
_1)\cdot m_1^*+\sum\limits_{m_1\in \binom{y,z}{n}} [b_2(\nu)](m_1^*)\cdot m_1$\vspace{5pt}\\
&=&$\sum\limits_{m_1\in \binom{y,z}{n-1}} \mathfrak b(\nu\wedge m_1)
\cdot m_1^*+\sum\limits_{m_1\in \binom{y,z}{n}} \mathfrak b(\nu\wedge m_1^*)\cdot m_1$
\vspace{5pt}\\
&=&$\begin{cases}\phantom{+}\sum\limits_{m_1\in \binom{y,z}{n-1}} 
\left(\vphantom{{\widetilde{\Phi}}}-\beta_2(m_1\otimes \nu)-\Psi(y)\cdot\pmb\delta\cdot [zm_1](\nu)+\Psi(z)\cdot\pmb\delta\cdot [ym_1](\nu)\right)
\cdot m_1^*\vspace{5pt}\\+\sum\limits_{m_1\in \binom{y,z}{n}} \beta_1(\nu\wedge m_1^*)\cdot m_1\end{cases}$
\vspace{5pt}\\
&=&$\begin{cases}\phantom{+}\sum\limits_{m_1\in \binom{y,z}{n-1}} 
-
\Psi(x)\cdot \left[\mathfrak Q((zm_1)({\widetilde{\Phi}})\otimes y(\nu))-\mathfrak Q((ym_1)({\widetilde{\Phi}})\otimes z(\nu))\right]\cdot m_1^* \vspace{5pt}\\
-\Psi(y)\cdot\pmb\delta\cdot z(\nu)+\Psi(z)\cdot\pmb\delta\cdot y(\nu)
\vspace{5pt}\\+\sum\limits_{m_1\in \binom{y,z}{n}} 
\left(\vphantom{{\widetilde{\Phi}}}\Psi(x)\cdot \mathfrak Q(z(\nu)\otimes y(m_1^*))-\Psi(x)\cdot \mathfrak Q(y(\nu)\otimes z(m_1^*))
\right)
\cdot m_1\end{cases}$
\vspace{5pt}\\
&=&$\begin{cases} -y\otimes\pmb\delta\cdot z(\nu)+z\otimes\pmb\delta\cdot y(\nu)
\vspace{5pt}\\
-x\otimes \sum\limits_{m_1\in \binom{y,z}{n-1}} 
\left[\mathfrak Q((zm_1)({\widetilde{\Phi}})\otimes y(\nu))-\mathfrak Q((ym_1)({\widetilde{\Phi}})\otimes z(\nu))\right]\cdot m_1^* \vspace{5pt}\\
+x\otimes \sum\limits_{m_1\in \binom{y,z}{n}} 
\left(\vphantom{{\widetilde{\Phi}}} \mathfrak Q(z(\nu)\otimes y(m_1^*))-\mathfrak Q(y(\nu)\otimes z(m_1^*))
\right)
\cdot m_1.\end{cases}$
\end{longtable}
\end{proof}

\begin{lemma}\label{comm-diag} The ${\mathfrak R}$-module homomorphisms $\tau_i:B_i\to G_i$ of Definition {\rm\ref{tau}} give rise to a commutative diagram$:$

$$\xymatrix{0\ar[r]&{\mathfrak R}\ar[r]^{b_3}\ar[d]^{\tau_3}&B_2\ar[r]^{b_2}\ar[d]^{\tau_2}&B_1\ar[r]^{b_1}\ar[d]^{\tau_1}&{\mathfrak R}\ar[d]^{=}\vspace{5pt}\\
0\ar[r]&G_3\ar[r]^{g_3}&G_2\ar[r]^{g_2}&G_1\ar[r]^{g_1}&{\mathfrak R}.}$$

\end{lemma}

\begin{proof} We first show that $b_1=g_1\circ\tau_1$.
If $\nu\in D_{n-1}^{{\mathbb Z}}(U_0^*)$ and $\mu\in {\operatorname{Sym}}_n^{{\mathbb Z}}U$, then, according to Definition~\ref{main-def}.\ref{main-def-c},
\begin{equation}\label{right-side'}b_1(\nu+\mu)=\Psi(x)\cdot \Psi(q(\nu))+\pmb\delta \cdot \Psi(\mu)-\Psi(x)\cdot \Psi(q(\mu({\widetilde{\Phi}}))).\end{equation}
On the other hand, $g_1$ is the restriction to $G_1$ of $1\otimes \Psi:{\mathfrak R}\otimes_{{\mathbb Z}} {\operatorname{Sym}}_{n}^{{\mathbb Z}}U\to {\mathfrak R}$; so, $(g_1\circ \tau_1)(\nu+\mu)$ is given by the right side of (\ref{right-side'}).

We show that $g_2\circ \tau_2=\tau_1\circ b_2$  for elements from each of the summands of the  module $$B_2={\mathfrak R}\otimes_{{\mathbb Z}}({\operatorname{Sym}}_{n-1}^{{\mathbb Z}}U_0\oplus D_{n}^{{\mathbb Z}} (U_0^*)).$$  
We first take $\mu\in  {\operatorname{Sym}}_{n-1}^{{\mathbb Z}}U_0$. Routine calculations yield
\begin{equation}\label{12.1.1}(g_2\circ \tau_2)(\mu)=\pmb\delta \begin{cases}
 \phantom{+}y\otimes \left(\pmb \delta z\mu-xq([z\mu]({\widetilde{\Phi}}))\right)
-z\otimes \left(\pmb \delta y\mu-xq([y\mu]({\widetilde{\Phi}}))\right)\vspace{5pt}\\
+x \otimes \left(yq((z\mu)({\widetilde{\Phi}}))-zq([y\mu]({\widetilde{\Phi}})) \right).\end{cases}\end{equation}
Indeed,
\begin{longtable}{lll}$(g_2\circ \tau_2)(\mu)$&$=$&$g_2\left(\pmb\delta\kappa(x\wedge z\otimes q((y\mu)({\widetilde{\Phi}})))-\pmb\delta\kappa (x \wedge y \otimes q((z\mu)({\widetilde{\Phi}})))-\pmb \delta^2 \kappa (y\wedge z\otimes \mu)\right)$\vspace{5pt}\\
&$=$&$\pmb\delta g_2\left(\begin{array}{lll}
\phantom{-}z\otimes xq((y\mu)({\widetilde{\Phi}}))&-x\otimes zq((y\mu)({\widetilde{\Phi}}))\vspace{5pt}\\
-y \otimes xq((z\mu)({\widetilde{\Phi}})))&+x \otimes yq((z\mu)({\widetilde{\Phi}})))\vspace{5pt}\\
-\pmb \delta  (z\otimes y\mu)&+\pmb \delta (y\otimes z\mu)&\text{in ${\mathfrak R}\otimes_{{\mathbb Z}} U\otimes_{{\mathbb Z}} {\operatorname{Sym}}_n^{{\mathbb Z}}U$}\end{array}\right)$\vspace{5pt}\\
&$=$&$\pmb\delta \left(\begin{array}{lll}
\phantom{-}z\otimes xq((y\mu)({\widetilde{\Phi}}))&-x\otimes zq((y\mu)({\widetilde{\Phi}}))\vspace{5pt}\\
-y \otimes xq((z\mu)({\widetilde{\Phi}})))&+x \otimes yq((z\mu)({\widetilde{\Phi}})))\vspace{5pt}\\
-\pmb \delta  (z\otimes y\mu)&+\pmb \delta (y\otimes z\mu)&\text{in ${\mathfrak R}\otimes_{{\mathbb Z}} {\operatorname{Sym}}_n^{{\mathbb Z}}U$}\end{array}\right).$\end{longtable}
Use  Observation~\ref{b2} to see that 
$(\tau_1\circ b_2)(\mu)$ is equal to
\begin{equation}\label{12.1.2} 
\begin{cases}
\phantom{+}
x\otimes \sum\limits_{m_1\in \binom{y,z}{n-1}}
\left [\mathfrak Q((z\mu)({\widetilde{\Phi}})\otimes (ym_1)({\widetilde{\Phi}}))-\mathfrak Q((y\mu)({\widetilde{\Phi}})\otimes (zm_1)({\widetilde{\Phi}}))\right]\cdot xq(m_1^*)\vspace{5pt}\\
+x\otimes\sum\limits_{m_1\in \binom{y,z}{n}}
\left[
\mathfrak Q((z\mu)({\widetilde{\Phi}})\otimes y(m_1^*))-\mathfrak Q((y\mu)({\widetilde{\Phi}})\otimes z(m_1^*))\right]
\cdot [\pmb\delta m_1-xq(m_1({\widetilde{\Phi}}))]\vspace{5pt}\\
+\pmb\delta y\otimes(\pmb\delta z\mu-xq([z\mu]({\widetilde{\Phi}})))-\pmb\delta z\otimes(\pmb\delta y\mu-xq([y\mu]({\widetilde{\Phi}}))).
\end{cases}
\end{equation}
Compare (\ref{12.1.1}) and (\ref{12.1.2}). It suffices to show that the elements
 $$\begin{array}{lll}X&=&\pmb \delta\left(yq([z\mu]({\widetilde{\Phi}}))
- zq([y\mu]({\widetilde{\Phi}}))\right)\quad\text{and}\vspace{5pt}\\ 
Y&=&\begin{cases}
+\sum\limits_{m_1\in \binom{y,z}{n-1}}
\left[\mathfrak Q((z\mu)({\widetilde{\Phi}})\otimes (ym_1)({\widetilde{\Phi}}))-\mathfrak Q((y\mu)({\widetilde{\Phi}})\otimes (zm_1)({\widetilde{\Phi}}))\right]\cdot xq(m_1^*)\vspace{5pt}\\
+\sum\limits_{m_1\in \binom{y,z}{n}}
\left[
\mathfrak Q((z\mu)({\widetilde{\Phi}})\otimes y(m_1^*))-\mathfrak Q((y\mu)({\widetilde{\Phi}})\otimes z(m_1^*))
\right]
\cdot [\pmb\delta m_1-xq(m_1({\widetilde{\Phi}}))]
\end{cases}\end{array}$$ of ${\mathfrak R}\otimes_{{\mathbb Z}}{\operatorname{Sym}}_{n}^{{\mathbb Z}} U$ are equal. To that end, we compare 
$X(\nu)$ and $Y(\nu)$ for $\nu\in D_n^{{\mathbb Z}}(U^*)$. Furthermore, it will suffice to compare   
$X(\nu)$ and $Y(\nu)$ for $\nu=m_2^*$ with $m_2\in \binom{y,z}{n}$ and $\nu=(xm_2)^*$ for $m_2\in \binom{x,y,z}{n-1}$ because $\binom{x,y,z}{n}=\binom{y,z}{n}\cup x\binom{x,y,z}{n-1}$ and $\{m^*\mid m\in \binom{x,y,z}{n}\}$ is a basis for $D_n^{{\mathbb Z}}(U^*)$. 

First take $\nu=m_2^*$ with $m_2\in \binom{y,z}{n}$. Observe that $$[xq(m_1^*)](m_2^*)=0, \quad [xq(m_1({\widetilde{\Phi}}))](m_2^*)=0,\quad\text{and}\quad \textstyle \sum\limits_{m_1\in \binom{y,z}{n}} m_1((m_2^*))\cdot m_1^*=m_2^*.$$ It follows that $$Y(m_2^*)=\pmb \delta\left(\mathfrak Q((z\mu)({\widetilde{\Phi}})\otimes y(m_2^*))-\mathfrak Q((y\mu)({\widetilde{\Phi}})\otimes z(m_2^*))\right)=X(m_2^*).$$

Now take  $\nu=(xm_2)^*$ for $m_2\in \binom{x,y,z}{n-1}$. We see that 
$$Y(\nu)=\begin{cases}
+\sum\limits_{m_1\in \binom{y,z}{n-1}}
\left[\mathfrak Q((z\mu)({\widetilde{\Phi}})\otimes (ym_1)({\widetilde{\Phi}}))-\mathfrak Q((y\mu)({\widetilde{\Phi}})\otimes (zm_1)({\widetilde{\Phi}}))\right]\cdot [xq(m_1^*)]((xm_2)^*)\vspace{5pt}\\
+\sum\limits_{m_1\in \binom{y,z}{n}}
\left[
\mathfrak Q((z\mu)({\widetilde{\Phi}})\otimes y(m_1^*))-\mathfrak Q((y\mu)({\widetilde{\Phi}})\otimes z(m_1^*))
\right]
\cdot [\pmb\delta m_1-xq(m_1({\widetilde{\Phi}}))]((xm_2)^*).
\end{cases}$$
Use the facts $x((xm_2)^*)=m_2^*$ and $m_1(xm_2)^*=0$ to re-write $Y(\nu)$ as $Y(\nu)=\sum_{i=1}^4Y_i$, with
$$\begin{array}{ll}
Y_1=&\sum\limits_{m_1\in \binom{x,y,z}{n-1}}
\left[\mathfrak Q((z\mu)({\widetilde{\Phi}})\otimes (ym_1)({\widetilde{\Phi}}))-\mathfrak Q((y\mu)({\widetilde{\Phi}})\otimes (zm_1)({\widetilde{\Phi}}))\right]\cdot [q(m_1^*)](m_2^*),\vspace{5pt}\\
Y_2=&-\sum\limits_{m_1\in \binom{x,y,z}{n-2}}
\left[\mathfrak Q((z\mu)({\widetilde{\Phi}})\otimes (yxm_1)({\widetilde{\Phi}}))-\mathfrak Q((y\mu)({\widetilde{\Phi}})\otimes (zxm_1)({\widetilde{\Phi}}))\right]\cdot [q((xm_1)^*)](m_2^*),\vspace{5pt}\\
Y_3=&\sum\limits_{m_1\in \binom{x,y,z}{n}}
\left[
\mathfrak Q((z\mu)({\widetilde{\Phi}})\otimes y(m_1^*))-\mathfrak Q((y\mu)({\widetilde{\Phi}})\otimes z(m_1^*))
\right]
\cdot [-q(m_1({\widetilde{\Phi}}))](m_2^*), \text{ and}\vspace{5pt}\\
Y_4=&-\sum\limits_{m_1\in \binom{x,y,z}{n-1}}
\left[
\mathfrak Q((z\mu)({\widetilde{\Phi}})\otimes y((xm_1)^*))-\mathfrak Q((y\mu)({\widetilde{\Phi}})\otimes z((xm_1)^*))
\right]
\cdot [-q((xm_1)({\widetilde{\Phi}}))](m_2^*).
\end{array}$$
We know that \begin{equation}\label{fourier-trick}\sum\limits_{m_1\in \binom{x,y,z}{n-1}}[q(m_1^*)](m_2^*)\cdot m_1
=\sum\limits_{m_1\in \binom{x,y,z}{n-1}}[q(m_2^*)](m_1^*)\cdot m_1=q(m_2^*);$$ hence,
$$Y_1=\mathfrak Q((z\mu)({\widetilde{\Phi}})\otimes (yq(m_2^*))({\widetilde{\Phi}}))-\mathfrak Q((y\mu)({\widetilde{\Phi}})\otimes (zq(m_2^*))({\widetilde{\Phi}})).\end{equation}
A similar trick yields
$$Y_3=-
\left[
\mathfrak Q((z\mu)({\widetilde{\Phi}})\otimes y([q(m_2^*)]({\widetilde{\Phi}})))-\mathfrak Q((y\mu)({\widetilde{\Phi}})\otimes z([q(m_2^*)]({\widetilde{\Phi}})))
\right];$$ thus, $Y_1+Y_3=0$. In $Y_2$,
$$\begin{array}{ll}
&\mathfrak Q((z\mu)({\widetilde{\Phi}})\otimes (yxm_1)({\widetilde{\Phi}}))-\mathfrak Q((y\mu)({\widetilde{\Phi}})\otimes (zxm_1)({\widetilde{\Phi}}))\vspace{5pt}\\=&
\mathfrak Q((z\mu)({\widetilde{\Phi}})\otimes (ym_1)(\Phi))-\mathfrak Q((y\mu)({\widetilde{\Phi}})\otimes (zm_1)(\Phi))\vspace{5pt}\\
=&\pmb\delta \left([(z\mu)({\widetilde{\Phi}})](ym_1)-[(y\mu)({\widetilde{\Phi}})](zm_1)\right)=0;\end{array}$$and therefore, $Y_2=0$. We conclude that $Y(\nu)=Y_4$. As we simplify $Y_4$, we see that
$$[q((xm_1)({\widetilde{\Phi}}))](m_2^*)=[q((m_1)(\Phi))](m_2^*)=\pmb\delta m_1(m_2^*)=\pmb\delta \chi(m_1=m_2).$$ The final equality holds because $m_1$ and $m_2$ are both monomials in $\binom{x,y,z}{n-1}$. It follows that
$$Y(\nu)=Y_4=
\pmb \delta \left[
\mathfrak Q((z\mu)({\widetilde{\Phi}})\otimes y((xm_2)^*))-\mathfrak Q((y\mu)({\widetilde{\Phi}})\otimes z((xm_2)^*))
\right]=X((xm_2)^*)=X(\nu),$$
and $g_2\circ \tau_2=\tau_1\circ b_2$  for elements of the summand ${\mathfrak R}\otimes_{{\mathbb Z}} {\operatorname{Sym}}_{n-1}^{{\mathbb Z}}U_0$ of $B_2$.

Now we show that $g_2\circ \tau_2=\tau_1\circ b_2$  for elements $\nu\in B_2$ with $\nu\in D_n^{{\mathbb Z}}(U^*_0)$.
We compute
$$\begin{array}{l}\phantom{=}(g_2\circ \tau_2)(\nu)
=\pmb \delta g_2
\left(\kappa (x\wedge z\otimes q(y(\nu)))-\kappa (x\wedge y\otimes q(z(\nu)))\vphantom{{\widetilde{\Phi}}}\right)\vspace{5pt}\\
=\pmb \delta g_2\left(\left(
z\otimes xq(y(\nu))-x\otimes zq(y(\nu))
-y\otimes xq(z(\nu))+x\otimes yq(z(\nu))\vphantom{{\widetilde{\Phi}}}
\right)\in \textstyle\bigwedge^1_{{\mathbb Z}}U\otimes_{{\mathbb Z}}{\operatorname{Sym}}_n^{{\mathbb Z}}U\right)\vspace{5pt}\\
=\pmb \delta \left(
z\otimes xq(y(\nu))-x\otimes zq(y(\nu))
-y\otimes xq(z(\nu))+x\otimes yq(z(\nu))
\vphantom{{\widetilde{\Phi}}}\right)\in {\mathfrak R}\otimes_{{\mathbb Z}}{\operatorname{Sym}}_n^{{\mathbb Z}}U.\end{array}$$
Use Observation~\ref{b2} to see that 
$$(\tau_1\circ b_2)(\nu)=\left\{\begin{array}{l}
-y\pmb\delta\otimes xq(z(\nu))+z\pmb\delta\otimes xq(y(\nu))\vspace{5pt}\\
-x\otimes  \sum\limits_{m_1\in \binom{y,z}{n-1}}\left[\mathfrak Q((zm_1)({\widetilde{\Phi}})\otimes y(\nu))-\mathfrak Q((ym_1)({\widetilde{\Phi}})\otimes z(\nu))\right] \cdot xq(m_1^*)\vspace{5pt}\\
+x\otimes \sum\limits_{m_1\in \binom{y,z}{n}} \left[\mathfrak Q(z(\nu)\otimes y(m_1^*))- \mathfrak Q(y(\nu)\vphantom{{\widetilde{\Phi}}}\otimes z(m_1^*))\right]
\cdot [\pmb\delta\cdot m_1-xq(m_1({\widetilde{\Phi}}))].
\end{array}\right.$$
Compare $\pmb \delta(g_2\circ \tau_2)(\nu)$ and $(\tau_1\circ b_2)(\nu)$.
In order to prove that these two expressions are equal, 
it suffices to show that $X=Y$ for the elements $X=\pmb\delta(yq(z(\nu))-zq(y(\nu)))$
and

$$Y=\left\{\begin{array}{l}
-\sum\limits_{m_1\in \binom{y,z}{n-1}}\left[\mathfrak Q((zm_1)({\widetilde{\Phi}})\otimes y(\nu))-\mathfrak Q((ym_1)({\widetilde{\Phi}})\otimes z(\nu))\right] \cdot xq(m_1^*)\vspace{5pt}\\
+\sum\limits_{m_1\in \binom{y,z}{n}} \left[\mathfrak Q(z(\nu)\otimes y(m_1^*))- \mathfrak Q(y(\nu)\vphantom{{\widetilde{\Phi}}}\otimes z(m_1^*))\right]
\cdot [\pmb\delta\cdot m_1-xq(m_1({\widetilde{\Phi}}))]
\end{array}\right.$$ of  ${\mathfrak R}\otimes_{{\mathbb Z}} {\operatorname{Sym}}^{{\mathbb Z}}_nU$. Apply $m_2^*$, with $m_2\in\binom{y,z}{n}$.
We have $x(m_2^*)=0$; so, 
$$Y(m_2^*)=\sum\limits_{m_1\in \binom{y,z}{n}} \left[\mathfrak Q(z(\nu)\otimes y(m_1^*))- \mathfrak Q(y(\nu)\vphantom{{\widetilde{\Phi}}}\otimes z(m_1^*))\right]\cdot \pmb\delta\cdot m_1(m_2^*)$$
$$= \pmb\delta\left[\mathfrak Q(z(\nu)\otimes y(m_2^*))- \mathfrak Q(y(\nu)\vphantom{{\widetilde{\Phi}}}\otimes z(m_2^*))\right]=X(m_2^*).$$Apply $(xm_2)^*$ with $m_2\in \binom{x,y,z}{n-1}$. We know that $m_1((xm_2)^*)=0$ for $m_1\in \binom{y,z}{n}$; but $x((xm_2)^*)=m_2^*$. So,
$$Y((xm_2)^*)=\left\{\begin{array}{l}
-\sum\limits_{m_1\in \binom{y,z}{n-1}}\left[\mathfrak Q((zm_1)({\widetilde{\Phi}})\otimes y(\nu))-\mathfrak Q((ym_1)({\widetilde{\Phi}})\otimes z(\nu))\right] \cdot [q(m_1^*)](m_2^*)\vspace{5pt}\\
+\sum\limits_{m_1\in \binom{y,z}{n}} \left[\mathfrak Q(z(\nu)\otimes y(m_1^*))- \mathfrak Q(y(\nu)\vphantom{{\widetilde{\Phi}}}\otimes z(m_1^*))\right]
\cdot [-[q(m_1({\widetilde{\Phi}}))] (m_2^*)].
\end{array}\right.$$Re-write $Y((xm_2)^*)$ as $\sum_{i=1}^4Y_i$ with 
$$\begin{array}{lll}
Y_1&=&-\sum\limits_{m_1\in \binom{x,y,z}{n-1}}\left[\mathfrak Q((zm_1)({\widetilde{\Phi}})\otimes y(\nu))-\mathfrak Q((ym_1)({\widetilde{\Phi}})\otimes z(\nu))\right] \cdot [q(m_1^*)](m_2^*)\vspace{5pt}\\
Y_2&=&+\sum\limits_{m_1\in \binom{x,y,z}{n-2}}\left[\mathfrak Q((zxm_1)({\widetilde{\Phi}})\otimes y(\nu))-\mathfrak Q((yxm_1)({\widetilde{\Phi}})\otimes z(\nu))\right] \cdot [q((xm_1)^*)](m_2^*)\vspace{5pt}\\
Y_3&=&-\sum\limits_{m_1\in \binom{x,y,z}{n}} \left[\mathfrak Q(z(\nu)\otimes y(m_1^*))- \mathfrak Q(y(\nu)\vphantom{{\widetilde{\Phi}}}\otimes z(m_1^*))\right]
\cdot [q(m_1({\widetilde{\Phi}}))] (m_2^*)\vspace{5pt}\\
Y_4&=&+\sum\limits_{m_1\in \binom{x,y,z}{n-1}} \left[\mathfrak Q(z(\nu)\otimes y((xm_1)^*))- \mathfrak Q(y(\nu)\vphantom{{\widetilde{\Phi}}}\otimes z((xm_1)^*))\right]
\cdot [q((xm_1)({\widetilde{\Phi}}))] (m_2^*).
\end{array}$$Use the trick of (\ref{fourier-trick}) to see that $Y_1+Y_3=0$ and use the defining property $x({\widetilde{\Phi}})=\Phi$ of ${\widetilde{\Phi}}$ together with parts (\ref{obvious-a}) and (\ref{obvious-e}) of Observation~\ref{obvious} to see that $$Y_2=\sum\limits_{m_1\in \binom{x,y,z}{n-2}}\left[
\pmb\delta (zm_1)(y(\nu))-\pmb\delta (ym_1)(z(\nu))\right] \cdot [q((xm_1)^*)](m_2^*)=0.$$
Of course, $[q((xm_1)({\widetilde{\Phi}}))] (m_2^*)=\pmb\delta m_1(m_2^*)$; and therefore
$$\begin{array}{lll}Y((xm_2)^*)&=&Y_4=\sum\limits_{m_1\in \binom{x,y,z}{n-1}} \left[\mathfrak Q(z(\nu)\otimes y((xm_1)^*))- \mathfrak Q(y(\nu)\vphantom{{\widetilde{\Phi}}}\otimes z((xm_1)^*))\right]
\cdot \pmb\delta m_1(m_2^*)\vspace{5pt}\\
&=&\pmb\delta \left[\mathfrak Q(z(\nu)\otimes y((xm_2)^*))- \mathfrak Q(y(\nu)\vphantom{{\widetilde{\Phi}}}\otimes z((xm_2)^*))\right]=X((xm_2)^*).\end{array}$$
This completes the proof that $g_2\circ \tau_2=\tau_1\circ b_2$.

Finally, we prove that $g_3\circ \tau_3=\tau_2\circ b_3$. 
Observe that $(g_3\circ \tau_3)(1)$  and $(\tau_2\circ b_3)(1)$ both are elements of ${\mathfrak R}\otimes_{{\mathbb Z}}U\otimes {\operatorname{Sym}}_n^{{\mathbb Z}}U$. We prove that  $(g_3\circ \tau_3)(1)=(\tau_2\circ b_3)(1)$ by showing that
\begin{equation}\label{nu}[(g_3\circ \tau_3)(1)](1\otimes 1\otimes \nu) = [(\tau_2\circ b_3)(1)](1\otimes 1\otimes \nu)
\qquad\text{in }{\mathfrak R}\otimes_{{\mathbb Z}}U,\end{equation} for each  $\nu\in D_n^{{\mathbb Z}}U^*$.
 We see that the left side of (\ref{nu})
$$\begin{array}{ll}=&\pmb \delta^2\sum\limits_{m\in \binom{x,y,z}{n-1}}\left(\begin{array}{l}
\phantom{+}xm\otimes y\otimes [zq(m^*)](\nu)-ym\otimes x\otimes [zq(m^*)](\nu)-xm\otimes z\otimes [yq(m^*)](\nu)\vspace{5pt}\\
+zm\otimes x\otimes [yq(m^*)](\nu)+ym\otimes z\otimes [xq(m^*)](\nu)-zm\otimes y\otimes [xq(m^*)](\nu)\end{array}\right)\vspace{5pt}\\
=&\pmb \delta^2\left(\vphantom{\widetilde{\Phi}}[zq(y(\nu))-yq(z(\nu))]\otimes x
+[xq(z(\nu))-zq(x(\nu))]\otimes y
+[y(q(x(\nu))-xq(y(\nu))]\otimes z\right).\end{array}$$
On the other hand, the right side of (\ref{nu}) is 
\begin{longtable}{l}
$\phantom{=}\begin{cases}\phantom{+}\pmb\delta\sum\limits_{m\in \binom{y,z}{n-1}}xq(m^*)\otimes 
\left(\begin{array}{l}
\phantom{+}z\otimes [xq((ym)({\widetilde{\Phi}}))](\nu)-x\otimes [zq((ym)({\widetilde{\Phi}}))](\nu)\vspace{5pt}\\
-y\otimes [xq((zm)({\widetilde{\Phi}}))](\nu)+x\otimes [yq((zm)({\widetilde{\Phi}}))](\nu)\vspace{5pt}\\
-\pmb \delta  z\otimes [ym](\nu)+\pmb \delta  y\otimes [zm](\nu)\end{array}\right)\vspace{5pt}\\
+\pmb\delta\sum\limits_{m\in \binom{y,z}n}[\pmb \delta m -xq(m({\widetilde{\Phi}}))]\otimes 
 \left(\vphantom{{\widetilde{\Phi}}}\begin{array}{l} 
\phantom{+}z  \otimes [xq(y(m^*))](\nu)-x  \otimes [zq(y(m^*))](\nu)\vspace{5pt}\\
-y \otimes [xq(z(m^*))](\nu)+x \otimes [yq(z(m^*))](\nu)
\end{array}
\right)
\end{cases}$\vspace{5pt}\\

$=\pmb\delta\begin{cases}
\phantom{+}xq\left((y[q(x(\nu))])({\widetilde{\Phi}})\right)\otimes z
-xq\left((y[q(z(\nu))])({\widetilde{\Phi}})\right)\otimes x
\vspace{5pt}\\
-xq\left((z[q(x(\nu))])({\widetilde{\Phi}})\right)\otimes y
+xq\left((z[q(y(\nu))])({\widetilde{\Phi}})\right)\otimes x
\vspace{5pt}\\
-\pmb \delta  xq(y(\nu))\otimes z+\pmb \delta  xq(z(\nu))\otimes y\vspace{5pt}\\
+\pmb \delta yq(x(\nu)) \otimes z  -\pmb \delta yq(z(\nu)) \otimes x  
\vspace{5pt}\\
-\pmb \delta zq(x(\nu)) \otimes y+\pmb \delta zq(y(\nu)) \otimes x  
\vspace{5pt}\\
-xq([yq(x(\nu))]({\widetilde{\Phi}}))\otimes z +xq([yq(z(\nu))]({\widetilde{\Phi}}))\otimes x \vspace{5pt}\\
+xq([zq(x(\nu))]({\widetilde{\Phi}}))\otimes y-xq([zq(y(\nu))]({\widetilde{\Phi}}))\otimes x  
\end{cases}$\vspace{5pt}\\

$=\pmb \delta^2 \left(\vphantom{{\widetilde{\Phi}}}[zq(y(\nu))-yq(z(\nu))] \otimes x+[xq(z(\nu))-zq(x(\nu))]\otimes y+[yq(x(\nu))-xq(y(\nu))]\otimes z\right)$.
\end{longtable}
The two sides of (\ref{nu}) agree and the proof is complete.\end{proof}

\begin{lemma}\label{main-lemma}Retain the notation and hypotheses of Lemma~{\rm\ref{comm-diag}}. For $0\le i\le 3$, define $E_i=\tau_i(B_i)$ and for $1\le i\le 3$, let $e_i$ be the restriction of $g_i:G_i\to G_{i-1}$ to $E_i$.  The following statements hold.

\begin{enumerate}[\rm(a)]
\item\label{main-lemma-a}Each module $E_i$ is a free ${\mathfrak R}$-module.
\begin{enumerate}[\rm(1)]
\item\label{main-lemma-a-i}The elements 
\begin{equation}\label{mar-3-14.1'}\textstyle\{\tau_1(m^*)\mid m\in\binom{y,z}{n-1}\}\cup  \{\tau_1(m)\mid m\in\binom{y,z}{n}\}\end{equation} form a basis for $E_1$. 
\item\label{main-lemma-a-ii}The elements \begin{equation}\label{mar-3-14.1''}\textstyle\{\tau_2(m)\mid m\in\binom{y,z}{n-1}\}\cup  \{\tau_2(m^*)\mid m\in\binom{y,z}{n}\}\end{equation} form a basis for $E_2$.
\item\label{main-lemma-a-iii} The module $E_0$ is equal to $G_0={\mathfrak R}$.
\item\label{main-lemma-a-iv} The module $E_3$ is the free $\mathfrak R$-module $\pmb \delta^2G_3$.
\end{enumerate}
\item\label{main-lemma-b} For each $i$, with $0\le i\le 3$, $(E_i)_{\pmb \delta}=(G_i)_{\pmb \delta}$.
\item\label{main-lemma-c} Each ${\mathfrak R}$-module $\tau_i: B_i\to E_i$ is an isomorphism.
\item\label{main-lemma-d} The sequence of homomorphisms $$(\mathbb E,e):\quad   \xymatrix{0\ar[r]& E_3\ar[r]^{e_3}&E_2\ar[r]^{e_2}&E_1\ar[r]^{e_1}&E_0}
$$ is a complex  of free  ${\mathfrak R}$-modules.
\item\label{main-lemma-e} The sequence of homomorphisms $(\mathbb B,b)$ of Definition~{\rm\ref{main-def}.\ref{main-def-c}} is a complex of free  ${\mathfrak R}$-modules.
\item\label{main-lemma-f}  The ${\mathfrak R}$-module homomorphisms $\tau_i: B_i\to E_i$  give 
an isomorphism of complexes $(\mathbb B,b)\simeq (\mathbb E,e):$
$$\xymatrix{0\ar[r]&{\mathfrak R}\ar[r]^{b_3}\ar[d]^{\tau_3}&B_2\ar[r]^{b_2}\ar[d]^{\tau_2}&B_1\ar[r]^{b_1}\ar[d]^{\tau_1}&{\mathfrak R}\ar[d]^{=}\vspace{5pt}\\
0\ar[r]&E_3\ar[r]^{e_3}&E_2\ar[r]^{e_2}&E_1\ar[r]^{e_1}&{\mathfrak R}}.$$ Furthermore, the localizations $(E,e)_{\pmb\delta}$ and $(G,g)_{\pmb \delta}$ are equal.

\item\label{main-lemma-g} All of the assertions of Theorem~{\rm\ref{old-mvd}} hold for the explicitly constructed complex $(\mathbb B,b)$ in place of $(\mathbb G,g)$.

\end{enumerate}
\end{lemma}

\begin{proof}  
Assertions (\ref{main-lemma-a-iii}) and (\ref{main-lemma-a-iv}) are obvious. Assertion (\ref{main-lemma-b}) is also obvious when $i=0$ or $i=3$.

\medskip\noindent{\bf (\ref{main-lemma-a-i}) and (\ref{main-lemma-b}) for $i=1$.}   We prove 
(\ref{main-lemma-a-i}) 
and (\ref{main-lemma-b}) for $i=1$ by showing that $(G_1)_{\pmb \delta}$ is a free ${\mathfrak R}_{\pmb \delta}$-module with basis  (\ref{mar-3-14.1'}). The fact that $q:{\mathfrak R}_{\pmb \delta}\otimes_{{\mathbb Z}}D_{n-1}^{{\mathbb Z}}(U^*)\to {\mathfrak R}_{\pmb \delta}\otimes_{{\mathbb Z}}{\operatorname{Sym}}_{n-1}^{{\mathbb Z}}U$ is an ${\mathfrak R}_{\pmb\delta}$-module isomorphism guarantees that 
\begin{equation}\label{4.3.1}\textstyle \text{the elements $\{ q(m^*)\mid m\in \binom{x,y,z}{n-1}\}$ form a basis for ${\mathfrak R}_{\pmb \delta}\otimes_{{\mathbb Z}}{\operatorname{Sym}}_{n-1}^{\mathbb Z}U$.}\end{equation} It follows that   the elements 
\begin{equation}\label{basis1}\textstyle \{xq(m^*)\mid m\in \binom{x,y,z}{n-1}\}\cup \{m\mid m\in \binom{y,z}{n}\}\end{equation}
form a basis for ${\mathfrak R}_{\pmb \delta}\otimes_{{\mathbb Z}}{\operatorname{Sym}}_{n}U$. Thus,
\begin{equation}\label{basis2}\textstyle \{ xq(m^*)\mid m\in \binom{y,z}{n-1}\}\cup
\{ \pmb\delta m-xq(m({\widetilde{\Phi}}))\mid m\in \binom{y,z}{n}\}\cup
\{ xq((xm)^*)\mid m\in \binom{x,y,z}{n-2}\}
\end{equation} is a basis for 
${\mathfrak R}_{\pmb \delta}\otimes_{{\mathbb Z}}{\operatorname{Sym}}_{n}U$. 
(This step is legitimate, but a little complicated. We took the basis (\ref{basis1}); 
multiplied each element in the right-most set by a unit and added an element of the sub-module spanned by the left-most set, then we  
partitioned the left-most set into two subsets.)
At any rate, (\ref{basis2}) is a basis for ${\mathfrak R}_{\pmb \delta}\otimes_{{\mathbb Z}}{\operatorname{Sym}}_{n}U$ and (\ref{basis2}) 
 is the union of (\ref{mar-3-14.1'}) and \begin{equation}\label{mar-16}\textstyle\{ xq((xm)^*)\mid m\in \binom{x,y,z}{n-2}\}.\end{equation} Observe that $\underbar v_1$ gives a bijection between the set (\ref{mar-16})  and  the basis 
$$\textstyle \{\pmb \delta m^*\mid m\in\binom{x,y,z}{n-2}\}\text{ for } 
{\mathfrak R}_{\pmb \delta}\otimes_{{\mathbb Z}} D_{n-2}(U^*)={\mathfrak R}_{\pmb \delta}\otimes_{{\mathbb Z}} K_{0,n-2},$$ 
since $$\underbar v_1(xq((xm)^*))=[xq((xm)^*)](\Phi)=x([q((xm)^*)](\Phi))=\pmb\delta\cdot x((xm)^*)=\pmb \delta m^*.$$
Thus, the map $\underbar v_1:{\mathfrak R}_{\pmb \delta}\otimes_{{\mathbb Z}}{\operatorname{Sym}}_{n}U\to {\mathfrak R}_{\pmb \delta}\otimes_{{\mathbb Z}} D_{n-2}(U^*)$ takes the basis 
$$\textstyle {\rm(\ref{mar-3-14.1'})} \cup  {\rm(\ref{mar-16})}$$  for 
${\mathfrak R}_{\pmb \delta}\otimes_{{\mathbb Z}}{\operatorname{Sym}}_{n}U$, sends each element of {\rm(\ref{mar-3-14.1'})} to zero and carries 
{\rm(\ref{mar-16})} bijectively onto a basis for 
${\mathfrak R}_{\pmb \delta}\otimes_{{\mathbb Z}} D_{n-2}(U^*)$. We conclude that $(\ker \underbar v_1)_{\pmb \delta}$ is the free ${\mathfrak R}_{\pmb \delta}$ module with basis {\rm(\ref{mar-3-14.1'})} and this establishes (\ref{main-lemma-a-i}) 
and (\ref{main-lemma-b}) for $i=1$.

\medskip\noindent {\bf (\ref{main-lemma-a-ii}) and (\ref{main-lemma-b}) for $i=2$.} 
 We prove  (\ref{main-lemma-a-ii}) and (\ref{main-lemma-b}) for $i=2$ by showing that $(G_2)_{\pmb \delta}$ is a free ${\mathfrak R}_{\pmb \delta}$-module with basis  (\ref{mar-3-14.1''}). Our argument is similar to the proof of (\ref{main-lemma-a-i}) 
and (\ref{main-lemma-b}) for $i=1$
 in that we 
prove that (\ref{mar-3-14.1''}) together with
\begin{equation}\label{something-else}\begin{array}{l}
\textstyle \phantom{{}\cup{}}\{\kappa(x\wedge y\otimes q((xm)^*))\mid m\in \binom{x,y,z}{n-2}\}\cup\{\kappa(x\wedge z\otimes q((xm)^*))\mid m\in \binom{x,y,z}{n-2}\}\vspace{5pt}\\{}\cup \{\kappa(x\wedge y\otimes q((ym)^*))\mid m\in \binom{y,z}{n-2}\}
\end{array}\end{equation} 
forms a basis 
for ${\mathfrak R}_{\pmb\delta}\otimes_{{\mathbb Z}}L_{1,n}$ with the property that $\underbar v_2$ carries (\ref{something-else}) bijectively onto a basis for ${\mathfrak R}_{\pmb\delta}\otimes_{{\mathbb Z}}K_{1,n-2}$.
 The $\mathbb Z$-modules $L_{1,n}$ and $K_{1,n-2}$ are known to be free and have bases 
\begin{equation}\label{L-basis}\begin{array}{l}\textstyle
\phantom{{}\cup{}}\{\kappa(x\wedge y\otimes m)\mid m\in \binom{x,y,z}{n-1}\}\cup 
\{\kappa(x\wedge z\otimes m)\mid m\in \binom{x,y,z}{n-1}\}\vspace{5pt}\\{}\cup
\{\kappa(y\wedge z\otimes m)\mid m\in \binom{y,z}{n-1}\}\end{array}\end{equation}
and
\begin{equation}\label{K-basis}\begin{array}{l}\textstyle
\phantom{{}\cup{}}\{\eta(x\wedge y\otimes (xm)^*)\mid m\in \binom{x,y,z}{n-2}\}\cup 
\{\eta(x\wedge z\otimes (xm)^*)\mid m\in \binom{x,y,z}{n-2}\}\vspace{5pt}\\{}\cup
\{\eta(x\wedge y\otimes (ym)^*)\mid m\in \binom{y,z}{n-2}\},\end{array}\end{equation}
respectively, see, for example, \cite[(5.4) and (5.5)]{EKK} or \cite[Examples~2.1.3.h and 2.1.17.h]{W}.
The basis (\ref{L-basis}) for $L_{1,n}$ leads to the decomposition of ${\mathfrak R}_{\pmb\delta}\otimes_{{\mathbb Z}}L_{1,n}$ into the following direct sum of free ${\mathfrak R}_{\pmb\delta}$-modules:
\begin{equation}\label{decomp1}{\mathfrak R}_{\pmb\delta}\otimes_{{\mathbb Z}}L_{1,n}=\begin{cases}{\mathfrak R}_{\pmb\delta}\otimes_{{\mathbb Z}}\kappa(x\wedge y\otimes {\operatorname{Sym}}_{n-1}^{{\mathbb Z}}U)\oplus
{\mathfrak R}_{\pmb\delta}\otimes_{{\mathbb Z}}\kappa(x\wedge z\otimes {\operatorname{Sym}}_{n-1}^{{\mathbb Z}}U)\vspace{5pt}\\{} \oplus \bigoplus\limits_{m\in \binom{y,z}{n-1}}
{\mathfrak R}_{\pmb\delta} \kappa(y\wedge z\otimes m)\end{cases}\end{equation}
For each $m\in \binom{y,z}{n-1}$, recall the element 
$$\tau_2 (m)=\pmb\delta\kappa(x\wedge z\otimes q((ym)({\widetilde{\Phi}})))-\pmb\delta\kappa (x \wedge y \otimes q((zm)({\widetilde{\Phi}})))-\pmb \delta^2 \kappa (y\wedge z\otimes m)$$ of $E_2$.
 Notice that $\tau_2(m)$ is equal to the sum of a unit of ${\mathfrak R}_{\pmb\delta}$ times the basis vector $\kappa(y\wedge z\otimes m)$ from the third summand in (\ref{decomp1}) plus an element from the first two summands. It follows that $\{\tau_2(m)\mid m\in \binom{y,z}{n-1}\}$ generates a free submodule of ${\mathfrak R}_{\pmb\delta}\otimes_{{\mathbb Z}}L_{1,n}$ and 
\begin{equation}\label{decomp2}{\mathfrak R}_{\pmb\delta}\otimes_{{\mathbb Z}}L_{1,n}=\begin{cases}{\mathfrak R}_{\pmb\delta}\otimes_{{\mathbb Z}}\kappa(x\wedge y\otimes {\operatorname{Sym}}_{n-1}^{{\mathbb Z}}U)\oplus
{\mathfrak R}_{\pmb\delta}\otimes_{{\mathbb Z}}\kappa(x\wedge z\otimes {\operatorname{Sym}}_{n-1}^{{\mathbb Z}}U)\vspace{5pt}\\{} \oplus \bigoplus\limits_{m\in \binom{y,z}{n-1}}
{\mathfrak R}_{\pmb\delta} \tau_2(m).\end{cases}\end{equation}
Use (\ref{4.3.1}) to see that 
$$\textstyle \{\kappa(x\wedge y\otimes q(m^*))\mid m\in\binom{x,y,z}{n-1}\}\cup \{\kappa(x\wedge z\otimes q(m^*))\mid m\in\binom{x,y,z}{n-1}\}\cup\{\tau_2(m)\mid m\in \binom{y,z}{n-1}\}$$ is a basis for ${\mathfrak R}_{\pmb\delta}\otimes_{{\mathbb Z}}L_{1,n}$. Keep in mind that $\binom{x,y,z}{n-1}$ is the disjoint union $x\binom{x,y,z}{n-2} \cup \binom{y,z}{n-1}$ and that $\binom{y,z}{n-1}$ is the disjoint union $y\binom{y,z}{n-1}\cup\{z^{n-1}\}$ as well as the disjoint union $z\binom{y,z}{n-1}\cup\{y^{n-1}\}$. Thus, 
$$\begin{array}{l}\textstyle \phantom{{}\cup{}}\{\kappa(x\wedge y\otimes q((xm)^*))\mid m\in\binom{x,y,z}{n-2}\}\cup \{\kappa(x\wedge z\otimes q((xm)^*))\mid m\in\binom{x,y,z}{n-2}\}\vspace{5pt}\\{}
\cup\{ \kappa(x\wedge y\otimes q((ym)^*))\mid m\in\binom{y,z}{n-2}\}
\cup\{ \kappa(x\wedge z\otimes q((zm)^*))\mid m\in\binom{y,z}{n-2}\}\vspace{5pt}\\{}\cup \{\kappa(x\wedge y\otimes q((z^{n-1})^*)), \kappa(x\wedge z\otimes q((y^{n-1})^*))\}
\cup\{\tau_2(m)\mid m\in \binom{y,z}{n-1}\}\end{array}$$ is a basis for ${\mathfrak R}_{\pmb\delta}\otimes_{{\mathbb Z}}L_{1,n}$. We subtract each basis element in the set $\{ \kappa(x\wedge y\otimes q((ym)^*))\mid m\in\binom{y,z}{n-2}\}$ from the corresponding basis element in the set 
$\{ \kappa(x\wedge z\otimes q((ym)^*))\mid m\in\binom{y,z}{n-2}\}$ to see that

\begin{equation}\label{this-basis}\begin{array}{l}\textstyle \phantom{{}\cup{}}\{\kappa(x\wedge y\otimes q((xm)^*))\mid m\in\binom{x,y,z}{n-2}\}\cup \{\kappa(x\wedge z\otimes q((xm)^*))\mid m\in\binom{x,y,z}{n-2}\}\vspace{5pt}\\{}
\cup\{ \kappa(x\wedge y\otimes q((ym)^*))\mid m\in\binom{y,z}{n-2}\}\vspace{5pt}\\{}
\cup\{ \kappa(x\wedge z\otimes q((zm)^*))
-\kappa(x\wedge y\otimes q((ym)^*))
\mid m\in\binom{y,z}{n-2}\}\vspace{5pt}\\{}\cup \{\kappa(x\wedge y\otimes q((z^{n-1})^*)), \kappa(x\wedge z\otimes q((y^{n-1})^*))\}
\cup\{\tau_2(m)\mid m\in \binom{y,z}{n-1}\}\end{array}\end{equation} is a basis for ${\mathfrak R}_{\pmb\delta}\otimes_{{\mathbb Z}}L_{1,n}$. 
The union of the first three subset of (\ref{this-basis}), namely
$$\begin{array}{l}\textstyle \phantom{{}\cup{}}\{\kappa(x\wedge y\otimes q((xm)^*))\mid m\in\binom{x,y,z}{n-2}\}\cup \{\kappa(x\wedge z\otimes q((xm)^*))\mid m\in\binom{x,y,z}{n-2}\}
\vspace{5pt}\\{}\cup\{ \kappa(x\wedge y\otimes q((ym)^*))\mid m\in\binom{y,z}{n-2}\}, \end{array}$$
is the set we have called (\ref{something-else}). We reparameterize sets four  and five of (\ref{this-basis}), namely
\begin{equation}\label{four-five}\begin{array}{l}\textstyle \phantom{{}\cup{}}\{ \kappa(x\wedge z\otimes q((zm)^*))
-\kappa(x\wedge y\otimes q((ym)^*))
\mid m\in\binom{y,z}{n-2}\}\vspace{5pt}\\{}\cup \{\kappa(x\wedge y\otimes q((z^{n-1})^*)), \kappa(x\wedge z\otimes q((y^{n-1})^*))\}.\end{array}\end{equation} If $m\in\binom{y,z}{n-2}$, then let $m_1$ be the monomial $yzm$ in $\binom{y,z}{n}$.
Observe that 
$$\begin{array}{lll}\tau_2(m_1^*)&=&\pmb\delta\kappa(x\wedge z\otimes q(y(m_1^*)))
-\pmb\delta\kappa(x\wedge y\otimes q(z(m_1^*)))\vspace{5pt}\\&=&\pmb\delta\kappa(x\wedge z\otimes q((zm)^*)))
-\pmb\delta\kappa(x\wedge y\otimes q((ym)^*))).\end{array}$$
Observe also that 
$$\tau_2(m_1^*)=\pmb\delta\kappa(x\wedge z\otimes q(y(m_1^*)))
-\pmb\delta\kappa(x\wedge y\otimes q(z(m_1^*)))=\begin{cases} \phantom{-}\pmb\delta\kappa(x\wedge z\otimes q((y^{n-1})^*)&\text{if $m_1=y^n$} \vspace{5pt}\\-\pmb\delta(\kappa(x\wedge y\otimes q((z^{n-1})^*)&\text{if $m_1=z^n$}. \end{cases}$$
Thus, $\{ \tau_2(m_1^*)
\mid m_1\in \binom{y,z}{n}\}$ is a basis for the free ${\mathfrak R}_{\pmb \delta}$-module spanned by (\ref{four-five}) and  
 $$\textstyle \text{the union of (\ref{something-else}), $\{ \tau_2(m_1^*)
\mid m_1\in \binom{y,z}{n}\}$, and $\{\tau_2(m)\mid m\in \binom{y,z}{n-1}\}$}$$ is a basis for ${\mathfrak R}_{\pmb\delta}\otimes_{{\mathbb Z}}L_{1,n}$. Furthermore,  the set  (\ref{mar-3-14.1''}) is the union of  
$$\textstyle \{ \tau_2(m_1^*)
\mid m_1\in \binom{y,z}{n}\}\cup\{\tau_2(m)\mid m\in \binom{y,z}{n-1}\}.$$ Therefore, we have established that the union of
 (\ref{mar-3-14.1''}) and (\ref{something-else}) is a basis for the free-module ${\mathfrak R}_{\pmb\delta}\otimes_{{\mathbb Z}}L_{1,n}$. We saw in Observation \ref{really} that each element in (\ref{mar-3-14.1''}) is in the kernel of $\underbar v_2$.
A straightforward calculation  shows that $\underbar v_2$ carries (\ref{something-else}) bijectively onto the  unit $\pmb \delta$ times the basis (\ref{K-basis})  of ${\mathfrak R}_{\pmb\delta}\otimes_{{\mathbb Z}}K_{1,n-2}$:
\begin{equation}\label{something-else-2}\begin{array}{llll}
\underbar v_2(\kappa(x\wedge y\otimes q((xm)^*)))&=&\pmb\delta\eta(x\wedge y\otimes (xm)^*) 
&\text{for }m\in \binom{x,y,z}{n-2}\vspace{5pt}\\
\underbar v_2(\kappa(x\wedge z\otimes q((xm)^*)))&=&\pmb\delta\eta(x\wedge z\otimes (xm)^*) 
&\text{for }m\in \binom{x,y,z}{n-2}\vspace{5pt}\\
\underbar v_2(\kappa(x\wedge y\otimes q((ym)^*))&=&\pmb\delta\eta(x\wedge y\otimes (ym)^*)  
&\text{for }m\in \binom{y,z}{n-2}.
\end{array}\end{equation} 
In particular, for example, the top equation in (\ref{something-else-2}) is:
$$\begin{array}{lll}\phantom{=}&\underbar v_2\left(\kappa(x\wedge y\otimes q((xm)^*))\vphantom{{\widetilde{\Phi}}}\right)\vspace{5pt}\\
=&\underbar v_2\left(y\otimes xq((xm)^*)-x\otimes yq((xm)^*) \vphantom{{\widetilde{\Phi}}}\right) &\text{the definition of $\kappa$}\vspace{5pt}\\
=&y\otimes [xq((xm)^*)](\Phi)-x\otimes [yq((xm)^*) )](\Phi)&\text{the definition of $\underbar v_2$}\vspace{5pt}\\
=&y\otimes x\left([q((xm)^*)](\Phi)\vphantom{{\widetilde{\Phi}}}\right)-x\otimes y\left([q((xm)^*) )](\Phi)\vphantom{{\widetilde{\Phi}}}\right)&\text{(\ref{mod-act})}\vspace{5pt}\\
=&\pmb\delta \left(y\otimes x((xm)^*)-x\otimes y((xm)^*) \vphantom{{\widetilde{\Phi}}}\right)&\text{Observation \ref{obvious}, items (\ref{obvious-a}) and (\ref{obvious-c})}\vspace{5pt}\\
=&\pmb\delta \eta(x\wedge y\otimes (xm)^*)
&\text{the definition of $\eta$}.\vspace{5pt}\\
\end{array}$$
We conclude that the kernel of $\underbar v_2:{\mathfrak R}_{\pmb\delta}\otimes_{{\mathbb Z}}L_{1,n}\to{\mathfrak R}_{\pmb\delta}\otimes_{{\mathbb Z}}K_{1,n-2}$ is the free ${\mathfrak R}_{\pmb\delta}$-module with basis (\ref{mar-3-14.1''}) and this 
completes the proof of  (\ref{main-lemma-a-ii}) and (\ref{main-lemma-b}) for $i=2$.

\medskip\noindent{\bf(\ref{main-lemma-c}).} We proved in (\ref{main-lemma-a}) that $\tau_i$ carries a basis for the free ${\mathfrak R}$-module $B_i$ bijectively onto a basis for the free ${\mathfrak R}$-module $E_i$.

\medskip\noindent{\bf (\ref{main-lemma-d}).} We must verify that $e_i(E_i)\subseteq E_{i-1}$ and this follows from (\ref{main-lemma-c}) and Lemma~\ref{comm-diag}:
$$e_i(E_i)=g_i(E_i)=g_i(\operatorname{im} \tau_i)=\operatorname{im}(g_i\circ \tau_i)=\operatorname{im}(\tau_{i-1}\circ b_i)=\tau_{i-1}\operatorname{im}(b_i)\subseteq \tau_{i-1}(B_{i-1})=E_{i-1}.$$

\medskip\noindent{\bf (\ref{main-lemma-e}).} We must verify that $b_i\circ b_{i+1}=0$. One may apply the fact that  $\tau_{i-1}$ is injective, together with  Lemma~\ref{comm-diag}, to the complex $(\mathbb G,g)$, in order  to see that  $$\tau_{i-1}\circ b_i\circ b_{i+1}=g_i\circ g_{i+1}\circ \tau_{i+1}=0;$$ hence, $b_i\circ b_{i+1}=0$.

\medskip\noindent{\bf (\ref{main-lemma-f}).} We know from (\ref{main-lemma-d}) and (\ref{main-lemma-e})
that $(\mathbb E,e)$ and $(\mathbb B,b)$ are complexes; from Lemma \ref{comm-diag} that $\tau:\mathbb B\to\mathbb E$ is a map of complexes; from (\ref{main-lemma-c}) that $\tau:\mathbb B\to\mathbb E$ is an isomorphism  of complexes; and from (\ref{main-lemma-b}) that $\mathbb E_{\pmb \delta}=\mathbb G_{\pmb \delta}$.

\medskip\noindent{\bf (\ref{main-lemma-g}).} We see in (\ref{main-lemma-f}) that $(\mathbb B,b)$ to isomorphic to a free sub-complex of $(\mathbb G,g)$ and that $(\mathbb B,b)_{\pmb \delta}$ and $(\mathbb G,g)_{\pmb \delta}$ are isomorphic complexes.
\end{proof}

\section{The matrix description of $\mathbb B$.}\label{mat-desc}

\setcounter{equation}{0}

Start with Data \ref{Opening-Data}. Pick a basis $x,y,z$ for $U$ and use the basis $\binom{x,y,z}{2n-2}$ for ${\operatorname{Sym}}_{2n-2}^{\mathbb Z}U$. It follows that ${\mathfrak R}$ is the bi-graded polynomial ring
$$\textstyle\mathbb Z[\Psi(x),\Psi(y),\Psi(z),\{\Phi(m)\mid  \binom{x,y,z}{2n-2}\} ].$$ 
The symbols $\Psi(x)$, $\Psi(y)$, $\Psi(z)$, and $\Phi(m)$ are all fairly cumbersome. In order to avoid these symbols, we write ${\mathrm{R}}$ in place of ${\mathfrak R}$ when we emphasize that we have chosen the monomial bases 
for $U$ and ${\operatorname{Sym}}_{2n-2}^{\mathbb Z}U$. Furthermore, 
we write 
\begin{equation}\label{in-place}\begin{array}{l}\text{$x$, $y$, $z$, and $t_{m}$ in ${\mathrm{R}}$ in place of $\Psi(x)$, $\Psi(y)$, $\Psi(z)$, and $\Phi(m)$ in ${\mathfrak R}$, respectively,}\vspace{5pt}\\ \text{for $m\in \binom{x,y,z}{2n-2}$.}\end{array}\end{equation} 
At any rate, ${\mathrm{R}}$ is the bi-graded polynomial ring
\begin{equation}\label{Ri}\textstyle {\mathrm{R}}=\mathbb Z[x,y,z,\{t_m|m\in \binom{x,y,z}{2n-2}\}],\end{equation}
where $x$, $y$, and $z$ have degree $(1,0)$ and each variable $t_{m}$ has degree $(0,1)$. The equation 
$$\Phi=\sum\limits_{m\in \binom{x,y,z}{2n-2}} t_m\otimes m^* \in {\mathrm{R}}\otimes_{\mathbb Z}D_{2n-2}^{{\mathbb Z}}(U^*)$$
is explained in Remark~\ref{Phi}. It follows immediately that the element ${\widetilde{\Phi}}$ of Data \ref{data-3} is given by \begin{equation}\label{superPhi}{\widetilde{\Phi}}=\sum\limits_{m\in \binom{x,y,z}{2n-2}} t_m\otimes (xm)^* \in {\mathrm{R}}\otimes_{\mathbb Z}D_{2n-1}^{{\mathbb Z}}(U^*).\end{equation}
In Proposition~\ref{explc}, we describe the complex $(\mathbb B,b)$ 
\begin{equation}\label{TBD} 0\to {\mathrm{R}} \xrightarrow{\ b_3\ } \begin{matrix} {\mathrm{R}}\otimes _{{\mathbb Z}} {\operatorname{Sym}}_{n-1}^{{\mathbb Z}}U_0\vspace{5pt}\\\oplus\vspace{5pt}\\
{\mathrm{R}}\otimes_{\mathbb Z} D_n^{{\mathbb Z}} (U_0^*)\end{matrix} \xrightarrow{\ b_2\ } 
\begin{matrix} {\mathrm{R}}\otimes _{{\mathbb Z}} D_{n-1}^{{\mathbb Z}}(U_0^*)\vspace{5pt}\\\oplus\vspace{5pt}\\
{\mathrm{R}}\otimes_{\mathbb Z} {\operatorname{Sym}}_n^{{\mathbb Z}} U_0\end{matrix} \xrightarrow{\ b_1\ } {\mathrm{R}}\end{equation}
in terms of the elements of ${\mathrm{R}}$ explicitly. As a bi-homogeneous complex $\mathbb B$ has the form
$$
0 \to {\textstyle {\mathrm{R}}\left(-2n-1,-3\binom{n+1}2+1\right)}\longrightarrow 
\begin{matrix}
{\mathrm{R}}\left(-n-1,-2\binom{n+1}2\right)^{n}\vspace{5pt}\\
\oplus\vspace{5pt}\\{\mathrm{R}}\left(-n-1,-2\binom{n+1}2+1\right)^{n+1}
\end{matrix}
\longrightarrow \begin{matrix}{\mathrm{R}}\left(-n,-\binom{n+1}2+1\right)^n\vspace{5pt}\\
\oplus\vspace{5pt}\\{\mathrm{R}}\left(-n,-\binom{n+1}2\right)^{n+1}\end{matrix} \longrightarrow {\mathrm{R}}.$$
There are two motivations for this project. First of all, we have promised that $(\mathbb B,b)$ is built in an explicit and polynomial manner from the coefficients of the  Macaulay inverse system $\Phi$; we are thereby compelled to leave no doubt that we have given an explicit description. Secondly, we recognize that some readers will prefer the description of Definition~\ref{main-def}; whereas others will prefer the description of Proposition~\ref{explc}.

\begin{remark}\label{10.1}Let $T$ be the matrix $(t_{m_1m_2})$ where $m_1$ and $m_2$ roam over  $\binom{x,y,z}{n-1}$ in the same order. Let $\pmb \delta$ be the determinant of $T$ and $Q$ be the classical adjoint of $T$. It makes sense to refer to the entries of $Q$ as $Q_{m_1,m_2}$ for $m_1$ and $m_2$ in $\binom{x,y,z}{n-1}$.

\begin{enumerate}[(a)] \item  Notice that $T$ is the matrix for $p$ from Data \ref{manuf-data}.\ref{8.19.d} with respect to the bases $\binom{x,y,z}{n-1}$ for ${\operatorname{Sym}}_{n-1}^{\mathbb Z}U$ and $\{m^*\mid m\in \binom{x,y,z}{n-1}\}$ for $D_{n-1}^{{\mathbb Z}}(U^*)$, because
$$\textstyle p({m_2})=\sum\limits_{m_1\in\binom{x,y,z}{n-1}}t_{m_1m_2}\otimes m_1^*\in {\mathrm{R}}\otimes_{{\mathbb Z}}D_{n-1}^{{\mathbb Z}}(U^*), \quad \text{for $m_2\in \binom{x,y,z}{n-1}$}. $$ \item The element $\pmb \delta\in {\mathrm{R}}$ of the present remark is the image  of $\pmb \delta\in {\mathfrak R}$ from Data \ref{manuf-data}.\ref{8.19.b} 
under the convention (\ref{in-place}).
 \item\label{10.1-c} Notice that $Q$ is the matrix for the map $q$ of Data \ref{manuf-data}.\ref{8.19.c} with respect to the bases $\{m^*\mid m\in \binom{x,y,z}{n-1}\}$ for $D_{n-1}^{{\mathbb Z}}(U^*)$ and $\binom{x,y,z}{n-1}$ for ${\operatorname{Sym}}_{n-1}^{\mathbb Z}U$, in the sense that
\begin{equation}\label{not-yet}\textstyle q({m_2^*})=\sum\limits_{m_1\in\binom{x,y,z}{n-1}}Q_{m_1,m_2}\otimes m_1\in {\mathrm{R}}\otimes_{{\mathbb Z}}{\operatorname{Sym}}_{n-1}^{{\mathbb Z}}U, \quad \text{for $m_2\in \binom{x,y,z}{n-1}$}.\end{equation}
\end{enumerate}
It follows from  Observation~\ref{obvious}.\ref{obvious-f},  Data~\ref{manuf-data}.\ref{manuf-data-d}, and (\ref{8.19.c}) that
\begin{equation}\label{Q-sub}\textstyle \mathfrak Q(m_1^*\otimes m_2^*)=\mathfrak Q(m_2^*\otimes m_1^*)=[q(m_2^*)](m_1^*)=Q_{m_1,m_2} \quad \text{for $m_1$ and $m_2$ in $\binom{x,y,z}{n-1}$}.\end{equation}
The order on  $\binom{x,y,z}{n-1}$ that was used to create the matrices $T$ and $Q$ is irrelevant; see, for example, Remark~\ref{square}. The matrices $T$ and $Q$ are both symmetric and, as was seen in Observation~\ref{obvious},   
\begin{equation}\label{TQ2}\textstyle \sum\limits_{m\in\binom{x,y,z}{n-1}} t_{m'm}Q_{m,m''}=\chi(m'=m'')\pmb \delta =\sum\limits_{m\in\binom{x,y,z}{n-1}} Q_{m'',m} t_{mm'},\quad \quad \text{ for all $m'$ and $m''$ in $\binom{x,y,z}{n-1}$.}\end{equation}
\end{remark}

\begin{definition}\label{10.2} For each $m_2\in \binom{x,y,z}{n-1}$, define the element $\lambda_{m_2}$ of ${\mathrm{R}}$ by 
$\lambda_{m_2} =\sum\limits_{m_1\in \binom{x,y,z}{n-1}} m_1Q_{m_1,m_2}$.\end{definition}

\begin{remark}\label{real-10.2} If $m_2\in \binom{x,y,z}{n-1}$, then
$$\begin{array}{rcll}(\Psi\circ q)(m_2^*)&=&\Psi(\sum\limits_{m_1\in\binom{x,y,z}{n-1}}Q_{m_1,m_2}\otimes m_1)&{\rm (\ref{not-yet})}\vspace{5pt}\\
&=&\sum\limits_{m_1\in\binom{x,y,z}{n-1}}\Psi(m_1)Q_{m_1,m_2}\in {\mathfrak R}\vspace{5pt}\\
&=&\sum\limits_{m_1\in\binom{x,y,z}{n-1}}m_1Q_{m_1,m_2}\in {\mathrm{R}}&{\rm(\ref{in-place})}\vspace{5pt}\\
&=&\lambda_{m_2}\in {\mathrm{R}}.&{\rm(\ref{10.2})}\end{array}$$
 \end{remark}

\begin{observation}\label{critical2} If $m_1\in \binom{y,z}n$, then 
$$m_1({\widetilde{\Phi}})=
\sum\limits_{m_2\in \binom{x,y,z}{n-2}} t_{m_1m_2} \otimes(xm_2)^*\in {\mathrm{R}}\otimes_{\mathbb Z} D_{n-1}^{\mathbb Z}(U^*).$$
\end{observation}
\begin{proof} The explicit form of ${\widetilde{\Phi}}$ is given in (\ref{superPhi}). The monomial $m_1$ does not involve $x$; so
$m_1[(xm)^*]$ is equal to $\chi(m_1|m)(x\frac m{m_1})^*$ and 
$$m_1({\widetilde{\Phi}})=\sum\limits_{m\in \binom{x,y,z}{2n-2}} t_m\otimes  m_1[(xm)^*]=
\sum\limits_{m\in \binom{x,y,z}{2n-2}} t_m \otimes \chi(m_1|m)(x{\textstyle\frac m{m_1}})^*.$$
Let $m_2=\chi(m_1|m)\frac m{m_1}$. Notice that as $m$ roams over $\binom{x,y,z}{2n-2}$, $m_2$ roams over $\binom{x,y,z}{n-2}$. The assertion follows. \end{proof}

\begin{proposition}\label{explc} The differentials in the complex $(\mathbb B,b)$ of Definition~{\rm\ref{main-def}.\ref{main-def-c}} are described below.
\begin{enumerate}[\rm (1)]
\item\label{explc-1} The ${\mathrm{R}}$-module homomorphism   $b_1$  is 
$$\begin{cases}
b_1(1\otimes m^*)=x\lambda_{m}, &\text{for $m\in \binom{y,z}{n-1}$, and}\vspace{5pt}\\
b_1(1\otimes m)=\pmb\delta m -x\sum\limits_{m_1\in \binom{x,y,z}{n-2}}\lambda_{xm_1} t_{m_1m},&\text{for $m\in \binom{y,z}{n}$}.\end{cases}$$
\item\label{explc-2} The ${\mathrm{R}}$-module homomorphism   $b_2$ is described below. 
\begin{enumerate}[\rm (a)]
\item\label{explc-2-a} If $m_2\in \binom{y,z}{n-1}$, then
$$b_2(1\otimes m_2) = \begin{cases}
\phantom{+}\sum\limits_{m_1\in \binom{y,z}{n-1}}\sum\limits_{M_1,M_2\in \binom{x,y,z}{n-2}}xQ_{xM_1,xM_2}
\det\bmatrix t_{m_1M_1y}&t_{m_1M_1z}\vspace{5pt}\\
t_{m_2M_2y}&t_{m_2M_2z}\endbmatrix \otimes m_1^*\vspace{5pt}\\
+\sum\limits_{m_1\in \binom{y,z}{n}} 
\sum\limits_{M\in \binom{x,y,z}{n-2}}x[\chi(y|m_1)Q_{\frac{m_1}y,xM}t_{Mm_2z}-\chi(z|m_1)Q_{\frac{m_1}z,xM}t_{Mm_2y}]\otimes m_1\vspace{5pt}\\
+y\pmb \delta\otimes zm_2-z\pmb\delta\otimes ym_2.
\end{cases}$$

\item \label{explc-2-b} If $m_2\in \binom{y,z}{n}$, then $b_2(1\otimes m_2^*)$ is equal to 

$$\begin{cases}
\phantom{+}\chi(z|m_2)\left[
\sum\limits_{m'\in \binom{x,y,z}{n-2}}\sum\limits_{M\in \binom{y,z}{n-1}}
xt_{yMm'}Q_{xm',\frac{m_2}{z}}\otimes M^*-y\pmb\delta \otimes (\frac{m_2}z)^*
+\sum\limits_{M\in \binom{y,z}{n-1}}xQ_{M,\frac{m_2}z}\otimes yM
\right]
\vspace{5pt}\\
-\chi(y|m_2)\left[
\sum\limits_{m'\in \binom{x,y,z}{n-2}}\sum\limits_{M\in \binom{y,z}{n-1}}
xt_{zMm'}Q_{xm',\frac{m_2}{y}}\otimes M^*-z\pmb\delta \otimes (\frac{m_2}y)^*
+\sum\limits_{M\in \binom{y,z}{n-1}}xQ_{M,\frac{m_2}y}\otimes zM
\right].

\end{cases}$$

\end{enumerate}

\item\label{explc-3} The ${\mathrm{R}}$-module homomorphism   $b_3$ is  given by
$$b_3(1)=\sum\limits_{m\in \binom{y,z}{n-1}}b_1(1\otimes m^*) \otimes m
+\sum\limits_{m\in \binom{y,z}{n}}b_1(1\otimes m) \otimes m^*.$$
\end{enumerate}
\end{proposition}

\begin{proof} We prove (\ref{explc-1}). Let $m_2\in \binom{y,z}{n-1}$. We see that
$$\begin{array}{rcll}
b_1(1\otimes m_2^*)&=& \Psi(x)\cdot \Psi(q(m_2^*))&(\text{\ref{main-def}.\ref{main-def-c}})\vspace{5pt}\\ 
&=&x\lambda_{m_2},&\text{(\ref{in-place}) and (\ref{real-10.2})}
\end{array}$$
as expected.
Let $m_2\in \binom{y,z}{n}$. We see that
$$\begin{array}{rcll}
b_1(1\otimes m_2)&=& \pmb\delta\cdot \Psi(m_2)-\Psi(x)\cdot (\Psi\circ q)(m_2({\widetilde{\Phi}}))&\text{(\ref{main-def}.\ref{main-def-c})}\vspace{5pt}\\
&=&\pmb\delta\cdot \Psi(m_2)-\Psi(x)\cdot \sum\limits_{m_1\in \binom{x,y,z}{n-2}}t_{m_1m_2}\otimes (\Psi\circ q)((xm_1)^*)&\text{(\ref{critical2})}\vspace{5pt}\\
&=&\pmb\delta m_2- x
\sum\limits_{m_1\in \binom{x,y,z}{n-2}}\lambda_{xm_1}t_{m_1m_2},&\text{(\ref{real-10.2}) and (\ref{in-place})}
\end{array}$$as expected. The proof of (\ref{explc-1}) is complete. Assertion (\ref{explc-3}) follows automatically.

We prove (\ref{explc-2-a}). Let $m_2\in \binom{y,z}{n-1}$. We use Observation~\ref{b2}.\ref{b2-a} to see that $b_2(1\otimes m_2)$ is equal to 

\begin{longtable}{ll}
$=\begin{cases}
\phantom{+}
\Psi(x) \sum\limits_{m_1\in \binom{y,z}{n-1}}
\left[\mathfrak Q((zm_2)({\widetilde{\Phi}})\otimes (ym_1)({\widetilde{\Phi}}))-\mathfrak Q((ym_2)({\widetilde{\Phi}})\otimes (zm_1)({\widetilde{\Phi}}))\right]
\otimes m_1^*
\vspace{5pt}\\
+\Psi(x) \sum\limits_{m_1\in \binom{y,z}{n}}
\left[\vphantom{{\widetilde{\Phi}}}\mathfrak Q((zm_2)({\widetilde{\Phi}})\otimes y(m_1^*))-\mathfrak Q((ym_2)({\widetilde{\Phi}})\otimes z(m_1^*))\right ]\otimes m_1\vspace{5pt}\\
+
\Psi(y)\cdot  \pmb \delta \otimes zm_2- \Psi(z)\cdot \pmb\delta \otimes ym_2.
\end{cases}

$\vspace{5pt}\\

$=\begin{cases} \sum\limits_{m_1\in \binom{y,z}{n-1}} 
\sum\limits_{M_1,M_2\in \binom{x,y,z}{n-2}}\Psi(x)\cdot [
t_{zm_2M_2}t_{ym_1M_1}
-t_{ym_2M_2}t_{zm_1M_1}]\mathfrak Q(
 (xM_2)^*
\otimes 
(xM_1)^*
)\otimes m_1^*\vspace{5pt}\\
+\sum\limits_{m_1\in \binom{y,z}{n}}\sum\limits_{M\in \binom{x,y,z}{n-2}}  
\Psi(x)\cdot [t_{zm_2M} \mathfrak Q((xM)^*\otimes y(m_1^*))
-t_{ym_2M}\mathfrak Q((xM)^*\otimes z(m_1^*))]
\otimes m_1\vspace{5pt}\\
+
\Psi(y)\cdot\pmb\delta\otimes zm_2- \Psi(z)\cdot\pmb\delta\otimes ym_2 \end{cases}$&\text{(\ref{critical2})}\vspace{5pt}\\

$=\begin{cases} \sum\limits_{m_1\in \binom{y,z}{n-1}} 
\sum\limits_{M_1,M_2\in \binom{x,y,z}{n-2}}\Psi(x)\cdot [
t_{zm_2M_2}t_{ym_1M_1}
-t_{ym_2M_2}t_{zm_1M_1}]Q_{xM_2,xM_1}
\otimes m_1^*\vspace{5pt}\\
+\sum\limits_{m_1\in \binom{y,z}{n}}\sum\limits_{M\in \binom{x,y,z}{n-2}}  
\Psi(x)\cdot [\chi(y|m_1)t_{zm_2M}  Q_{xM,\frac{m_1}y}
-\chi(z|m_1)t_{ym_2M} Q_{xM,\frac {m_1}z}]
\otimes m_1\vspace{5pt}\\
+
\Psi(y)\cdot\pmb\delta\otimes zm_2- \Psi(z)\cdot\pmb\delta\otimes ym_2 \end{cases}$&\text{(\ref{Q-sub})}\vspace{5pt}\\
$=\begin{cases} \sum\limits_{m_1\in \binom{y,z}{n-1}} 
\sum\limits_{M_1,M_2\in \binom{x,y,z}{n-2}}x[
t_{zm_2M_2}t_{ym_1M_1}
-t_{ym_2M_2}t_{zm_1M_1}]Q_{xM_2,xM_1}
\otimes m_1^*\vspace{5pt}\\
+\sum\limits_{m_1\in \binom{y,z}{n}}\sum\limits_{M\in \binom{x,y,z}{n-2}}  
x[\chi(y|m_1)t_{zm_2M}  Q_{xM,\frac{m_1}y}
-\chi(z|m_1)t_{ym_2M} Q_{xM,\frac {m_1}z}]
\otimes m_1\vspace{5pt}\\
+
y\pmb\delta\otimes zm_2- z\pmb\delta\otimes ym_2, \end{cases}$&\text{(\ref{in-place})}

\end{longtable}
\noindent as expected. We prove (\ref{explc-2-b}). Let $m_2\in \binom{y,z}{n}$. We use Observation~\ref{b2}.\ref{b2-b} to see that $b_2(1\otimes m_2^*)$ is equal to 

\begin{longtable}{ll}
$=\begin{cases} -\sum\limits_{m_1\in \binom{y,z}{n-1}} 
\Psi(x)\cdot [\mathfrak Q((zm_1)({\widetilde{\Phi}})\otimes y(m_2^*))-\mathfrak Q((ym_1)({\widetilde{\Phi}})\otimes z(m_2^*))]\otimes m_1^*\vspace{5pt}\\
+\sum\limits_{m_1\in \binom{y,z}{n}} 
\Psi(x)\cdot [\mathfrak Q(z(m_2^*)\otimes y(m_1^*))-\mathfrak Q(y(m_2^*)\otimes z(m_1^*))]
\otimes m_1\vspace{5pt}\\
-\Psi(y)\cdot\pmb\delta\otimes z(m_2^*)+ \Psi(z)\cdot\pmb\delta\otimes  y(m_2^*)
\end{cases}$&{\rm (\ref{main-def}.\ref{main-def-a})}\vspace{5pt}\\

$=\begin{cases} -\sum\limits_{m_1\in \binom{y,z}{n-1}} \sum\limits_{m \in \binom{x,y,z}{n-2}}
\Psi(x)\cdot [t_{zm_1m}\mathfrak Q(
  (xm )^*

\otimes y(m_2^*))-t_{ym_1m}\mathfrak Q(
  (xm )^*

\otimes z(m_2^*))]\otimes m_1^*\vspace{5pt}\\
+\sum\limits_{m_1\in \binom{y,z}{n}} 
\Psi(x)\cdot [\mathfrak Q(z(m_2^*)\otimes y(m_1^*))-\mathfrak Q(y(m_2^*)\otimes z(m_1^*))]
\otimes m_1\vspace{5pt}\\
-\Psi(y)\cdot\pmb\delta\otimes z(m_2^*)+ \Psi(z)\cdot\pmb\delta\otimes  y(m_2^*)
\end{cases}$&\text{(\ref{critical2})}\vspace{5pt}\\

$=\begin{cases} -\sum\limits_{m_1\in \binom{y,z}{n-1}} \sum\limits_{m \in \binom{x,y,z}{n-2}}
\Psi(x)\cdot [t_{zm_1m}\chi(y|m_2)Q_{xm,\frac{m_2}y}-t_{ym_1m} \chi(z|m_2)Q_{xm,\frac{m_2}z}]\otimes m_1^*\vspace{5pt}\\
+\sum\limits_{m_1\in \binom{y,z}{n-1}} 
\Psi(x)\cdot \chi(z|m_2) Q_{\frac{m_2}z,m_1}\otimes ym_1
-\sum\limits_{m_1\in \binom{y,z}{n-1}}\Psi(x)\cdot \chi(y|m_2)Q_{\frac{m_2}y,m_1}
\otimes zm_1\vspace{5pt}\\
-\Psi(y)\cdot\pmb\delta\otimes z(m_2^*)+ \Psi(z)\cdot\pmb\delta\otimes  y(m_2^*)
\end{cases}$&\text{(\ref{Q-sub})}\vspace{5pt}\\

$=\begin{cases} -\sum\limits_{m_1\in \binom{y,z}{n-1}} \sum\limits_{m \in \binom{x,y,z}{n-2}}
x[t_{zm_1m}\chi(y|m_2)Q_{xm,\frac{m_2}y}-t_{ym_1m} \chi(z|m_2)Q_{xm,\frac{m_2}z}]\otimes m_1^*\vspace{5pt}\\
+\sum\limits_{m_1\in \binom{y,z}{n-1}} 
x\chi(z|m_2) Q_{\frac{m_2}z,m_1}\otimes ym_1
-\sum\limits_{m_1\in \binom{y,z}{n-1}}x\chi(y|m_2)Q_{\frac{m_2}y,m_1}
\otimes zm_1\vspace{5pt}\\
-y\pmb\delta\otimes z(m_2^*)+ z\pmb\delta\otimes  y(m_2^*),
\end{cases}$&\text{(\ref{in-place})}
\end{longtable}
as expected. 
\end{proof}

\section{Examples}\label{ex}

Consider the resolution $(\mathbb B,b)$ of Definition~\ref{main-def}, with $n=3$. Let $x,y,z$ be a basis for $U$ and write $R=\mathbb Z[x,y,z]$ in place of ${\mathfrak R}_{(\bullet,0)}={\operatorname{Sym}}_{\bullet}^{\mathbb Z}(U)$.  For each index $i$, with $0\le i\le 3$, we exhibit an $R$-algebra homomorphism $\rho_i: {\mathfrak R} \to R$,  with the property that $\rho_i({\mathfrak R}_{(0,1)})\subseteq \mathbb Z$ and $\rho_i(\pmb \delta)\neq 0$  in $\mathbb Z$. For each $i$, we record the resolution $\rho_i\otimes_{{\mathfrak R}} \mathbb B_{\pmb \delta}$ of $R_{\pmb \delta}/I_i$ by free $R_{\pmb \delta}$-modules, where $I_i$ is the annihilator of $\Phi_i= \rho_i\otimes_{{\mathfrak R}} \Phi$.  We focus on the particular ideals $I_0,\dots,I_3$ because it is shown in \cite{EKK} that none of these four ideals may be obtained from another by way of change of variables. (Actually, the calculation in \cite{EKK} is made over a field of characteristic zero; hence, the conclusion also holds over $\mathbb Z_{\pmb \delta}$.) If one looks at the examples from the point of view of the presentation matrices $\rho_i\otimes b_2$, then the work in \cite{EKK} says that none of the presentation matrices  $\rho_i\otimes b_2$ may be obtained from another by performing a sequence of change of variables and invertible row and column operations. (This is the topic of Project~\ref{P4} from the Introduction.)

To describe the $R$-algebra homomorphism $\rho_i: {\mathfrak R} \to R$, it suffices to record $\Phi_i= \rho_i\otimes_{{\mathfrak R}} \Phi$ because ${\mathfrak R}$ is a polynomial ring over $R$;  each variable of ${\mathfrak R}$ over $R$ appears as a coefficient in $$\Phi=\sum_{m\in\binom{x,y,z}{2n-2}}\Phi(m)\otimes m^*\in {\mathfrak R}\otimes_{{\mathbb Z}} D_{2n-2}(U^*);$$ and ${\mathfrak R}$ is the polynomial ring $R[\{\Phi(m)\mid m\in\binom{x,y,z}{2n-2}\}]$. At any rate, 
\begin{equation}\label{phis}\begin{array}{lll}
\Phi_0&=&(x^2y^2)^*-(xyz^2)^*+2(z^4)^*+(x^4)^*+2(y^4)^{*},\vspace{5pt}\\
\Phi_1&=&(x^2y^2)^*-(xyz^2)^*+2(z^4)^*+(x^4)^*,\vspace{5pt}\\
\Phi_2&=&(x^2y^2)^*-(xyz^2)^*+2(z^4)^*,\text{ and}\vspace{5pt}\\ 
\Phi_3&=& (y^2z^2)^*+(x^2z^2)^*+(x^2y^2)^*+2(xyz^2)^*+2(xy^2z)^*+2(x^2yz)^*.
\end{array}\end{equation}The fact that none of the ideals $\{I_i\}$ may be obtained for any other ideal from this list by way of change of variables is due to the fact that there are $i$ linearly independent linear forms $\ell_1,\dots,\ell_i$ in $\mathbb Z[x,y,z]$ with $\ell_1^3,\dots,\ell_i^3$ in $I_i$ but there does not exist $i+1$ such linearly independent linear forms. (The existence of $\ell_1,\dots,\ell_i$ is clear in each case; the non-existence of $i+1$ such linear forms requires a calculation and this calculation is made in \cite[Prop.~7.14]{EKK}.) The ideal $I_2$ is generated by the maximal order Pfaffians of the Buchsbaum-Eisenbud  matrix \cite[Sect.~6, pg.~480]{BE77}$$\bmatrix 
0&x&0&0&0&0&z\vspace{5pt}\\
-x&0&y&0&0&z&0\vspace{5pt}\\
0&-y&0&x&z&0&0\vspace{5pt}\\
0&0&-x&0&y&0&0\vspace{5pt}\\
0&0&-z&-y&0&x&0\vspace{5pt}\\
0&-z&0&0&-x&0&y\vspace{5pt}\\
-z&0&0&0&0&-y&0\endbmatrix.$$ (A proof of this assertion is contained in \cite[Prop.~7.10]{EKK}.) The Macaulay inverse systems $\Phi_1$ and $\Phi_0$ are  modifications of $\Phi_2$. It is shown in Lemma~\ref{6.2} that the ideal $I_3$ is equal  \begin{equation}\label{6.0}(x^3,y^3,z^3):(x+y+z)^2.\end{equation} 

Now that the $\Phi_i$ are defined in (\ref{phis}), the $R$-algebra homomorphisms $\rho_i$ are implicitly defined, as described above (\ref{phis}). For each $i$, we record the matrices $T_i=\rho_i\otimes T$ and $Q_i=\rho_i\otimes Q$ for $T$ and $Q$ from Remark~\ref{10.1}. We express these matrices using the basis 
$x^2,xy,xz,y^2,yz,z^2$ for ${\operatorname{Sym}}_2(U)$. We also describe the homomorphisms $\rho_i\otimes b_1$ and $\rho_i\otimes b_2$ using the basis $(y^2)^*,(yz)^*,(z^2)^*,y^3,y^2z,yz^2,z^3$ for $\rho_i\otimes B_1$ and the basis $y^2,yz,z^2,(y^3)^*,(y^2z)^*,(yz^2)^*,(z^3)^*$ for $\rho_i\otimes B_2$.

 \begin{example} When $i=0$, then $\Phi_0= (x^2y^2)^*-(xyz^2)^*+2(z^4)^*+(x^4)^*+2(y^4)^{*}$,
$$T_0=\bmatrix
1& 0 & 0  &1 &0  &0  \vspace{5pt}\\
0& 1 & 0  &0 &0  &-1 \vspace{5pt}\\
0& 0 & 0  &0 &-1 &0  \vspace{5pt}\\
1& 0 & 0  &2 &0  &0  \vspace{5pt}\\
0& 0 & -1 &0 &0  &0  \vspace{5pt}\\
0& -1& 0  &0 &0  &2  \endbmatrix,\quad  Q_0=\bmatrix
 -2& 0 & 0& 1 & 0& 0  \vspace{5pt}\\
 0 & -2& 0& 0 & 0& -1 \vspace{5pt}\\
 0 & 0 & 0& 0 & 1& 0  \vspace{5pt}\\
 1 & 0 & 0& -1& 0& 0  \vspace{5pt}\\
 0 & 0 & 1& 0 & 0& 0  \vspace{5pt}\\
 0 & -1& 0& 0 & 0& -1 \vspace{5pt}\\
\endbmatrix,$$  $\rho_0(\pmb \delta)=-1$, 
$$\begin{array}{lll}
(\rho_0\otimes b_1)(1\otimes (y^2)^*)&=&x^3-xy^2,\vspace{5pt}\\
(\rho_0\otimes b_1)(1\otimes (yz)^*)&=&x^2z,\vspace{5pt}\\
(\rho_0\otimes b_1)(1\otimes (z^2)^*)&=&-x^2y-xz^2,\vspace{5pt}\\
(\rho_0\otimes b_1)(1\otimes y^3)&=&- y^3 +4x^2y+2xz^2,\vspace{5pt}\\
(\rho_0\otimes b_1)(1\otimes y^2z)&=&- y^2z,\vspace{5pt}\\
(\rho_0\otimes b_1)(1\otimes yz^2)&=&- yz^2 -2x^3+xy^2,\vspace{5pt}\\
(\rho_0\otimes b_1)(1\otimes z^3)&=&- z^3 -2xyz, \text{ and}\end{array}$$
\begin{equation}\label{star}\rho_0\otimes b_2=\left[\begin{array}{rrrrrrr}
0&0&0&-z&y&0&-2x\vspace{5pt}\\
0&0&2x&x&-z&y&0\vspace{5pt}\\
0&-2x&0&0&-3x&-z&y\vspace{5pt}\\
z&-x&0&0&-x&0&0\vspace{5pt}\\
-y&z&3x&x&0&0&0\vspace{5pt}\\
0&-y&z&0&0&0&-x\vspace{5pt}\\
2x&0&-y&0&0&x&0
\end{array}\right].\end{equation} We provide a few details in this example; however the calculations are straightforward and we suppress them in Examples \ref{i=1}, \ref{i=2}, and \ref{6.4}. The rows and columns of $T_0$ are tagged by the monomials 
$x^2,xy,xz,y^2,yz,z^2$ in that order. The entry in row $m_i$, column $m_j$ is $(m_im_j)(\Phi_0)$. So in particular, row $4$ of $T_0$ is $$[(y^2x^2)(\Phi_0),(y^2xy)(\Phi_0),(y^2xz)(\Phi_0),(y^2y^2)(\Phi_0),
(y^2yz)(\Phi_0),(y^2z^2)(\Phi_0)]=[1,0,0,2,0,0].$$
One computes $\rho_0(\pmb \delta)= \det T_0$ and the classical adjoint $Q_0$, which is equal to, $(\det T_0)$ times the inverse of $T_0$. The recipe of Definition~\ref{10.2} gives
$$\begin{array}{lll}
\rho_0(\lambda_{x^2})&=&-2x^2+y^2,\\
\rho_0(\lambda_{xy})&=&-2xy-z^2,\\
\rho_0(\lambda_{xz})&=&yz,\\
\rho_0(\lambda_{y^2})&=&x^2-y^2,\\
\rho_0(\lambda_{yz})&=&xz,\\
\rho_0(\lambda_{z^2})&=&-xy-z^2,\end{array}$$
and Proposition~\ref{explc}.\ref{explc-1} immediately gives   the value of $(\rho_0\otimes b_1)$ applied to $(y^2)^*$, $(yz)^*$, $(z^2)^*$. We compute
$$\begin{array}{lll}(\rho_0\otimes b_1)(1\otimes y^3)&=&\rho_0(\pmb\delta y^3 -x\sum\limits_{m_1\in \binom{x,y,z}{1}}\lambda_{xm_1} t_{m_1y^3})\vspace{5pt}\\
&=&- y^3 -x\rho_0(\lambda_{xx} t_{xy^3}+\lambda_{xy} t_{yy^3}+\lambda_{xz} t_{zy^3})\vspace{5pt}\\
&=&- y^3 -x((-2x^2+y^2) (0)+(-2xy-z^2) (2)+yz (0))\vspace{5pt}\\
&=&- y^3 -x(2(-2xy-z^2) )=- y^3 +4x^2y+2xz^2.\end{array}$$
One computes $(\rho_0\otimes b_1)$ of the basis elements $y^2z$, $yz^2$ and $z^3$ in a similar manner.

Use Proposition~\ref{explc}.\ref{explc-2-a} to see that $$(\rho_0\otimes b_2)(1\otimes y^2) = \rho_0\begin{cases}
\phantom{+}\sum\limits_{m_1\in \binom{y,z}{2}}\sum\limits_{M_1,M_2\in \binom{x,y,z}{1}}xQ_{xM_1,xM_2}
\det\bmatrix t_{m_1M_1y}&t_{m_1M_1z}\vspace{5pt}\\
t_{y^2M_2y}&t_{y^2M_2z}\endbmatrix \otimes m_1^*\vspace{5pt}\\
+\sum\limits_{m_1\in \binom{y,z}{3}} 
\sum\limits_{M\in \binom{x,y,z}{1}}x[\chi(y|m_1)Q_{\frac{m_1}y,xM}t_{My^2z}-\chi(z|m_1)Q_{\frac{m_1}z,xM}t_{My^2y}]\otimes m_1\vspace{5pt}\\
-y\otimes y^2z+z\otimes y^3.
\end{cases}$$
We have seen that $\rho_0(t_{y^2M_2y})=2\chi(M_2=y)$ and $\rho_0(t_{y^2M_2z})=0$. It follows that 

$$(\rho_0\otimes b_2)(1\otimes y^2) = \rho_0\begin{cases}
-2x\sum\limits_{m_1\in \binom{y,z}{2}}\sum\limits_{M_1\in \binom{x,y,z}{1}}Q_{xM_1,xy}
t_{m_1M_1z}
\otimes m_1^*\vspace{5pt}\\
-2x\sum\limits_{m_1\in \binom{y,z}{3}} 
\chi(z|m_1)Q_{\frac{m_1}z,xy}\otimes m_1\vspace{5pt}\\
-y\otimes y^2z+z\otimes y^3.
\end{cases}$$
We have seen that $\rho_0(Q_{xM_1,xy})=-2\chi(M_1=y)$ and if $m_1\in \binom{y,z}3$, then $$\rho_0(\chi(z|m_1)Q_{\frac{m_1}z,xy})=-\chi(m_1=z^3).$$ It follows that

$$(\rho_0\otimes b_2)(1\otimes y^2) = 4x\sum\limits_{m_1\in \binom{y,z}{2}}\rho_0(t_{m_1yz})
\otimes m_1^*
+2x\otimes z^3-y\otimes y^2z+z\otimes y^3.
$$
We see that $\rho_0(t_{m_1yz})=0$ for $m_1\in \binom{y,z}{2}$. We conclude that 
$$(\rho_0\otimes b_2)(1\otimes y^2) = 
2x\otimes z^3-y\otimes y^2z+z\otimes y^3.
$$ We have recorded this calculation as column one of (\ref{star}). One computes $\rho_0\otimes b_2$ of the basis elements $yz$ and $z^2$ in a similar manner. 

Use Proposition~\ref{explc}.\ref{explc-2-b} to see that $(\rho_0\otimes b_2)(1\otimes (y^3)^*)$
is equal to $$=-\rho_0\left[
\sum\limits_{m'\in \binom{x,y,z}{1}}\sum\limits_{M\in \binom{y,z}{2}}
xt_{zMm'}Q_{xm',y^2}\otimes M^*-z\pmb\delta \otimes (y^2)^*
+\sum\limits_{M\in \binom{y,z}{2}}xQ_{M,y^2}\otimes zM
\right].
$$
Observe that $\rho_0(Q_{xm',y^2})=\chi(m'=x)$ and if $M\in \binom{y,z}{2}$, then $\rho_0(Q_{M,y^2})=-\chi(M=y^2)$. It follows that 
$$
(\rho_0\otimes b_2)(1\otimes (y^3)^*)
=
-\rho_0\left[
\sum\limits_{M\in \binom{y,z}{2}}
xt_{zMx}\otimes M^*+z \otimes (y^2)^*
-x\otimes y^2z
\right].
$$
The value of $\rho_0(t_{zMx})$ is $-\chi(M=yz)$; thus,
$$(\rho_0\otimes b_2)(1\otimes (y^3)^*)=
-\left[
-x\otimes (yz)^*+z \otimes (y^2)^*
-x\otimes y^2z
\right].
$$  We have recorded this calculation as column four of (\ref{star}). 
One computes $\rho_0\otimes b_2$ of the basis elements $(y^2z)^*$, $(yz^2)^*$,  and $(z^3)^*$ in a similar manner. \end{example}
\begin{example}\label{i=1} When $i=1$, then $\Phi_1= (x^2y^2)^*-(xyz^2)^*+2(z^4)^*+(x^4)^*$,
$$T_1=\bmatrix
1& 0 & 0  &1 &0  &0  \vspace{5pt}\\
0& 1 & 0  &0 &0  &-1 \vspace{5pt}\\
0& 0 & 0  &0 &-1 &0  \vspace{5pt}\\
1& 0 & 0  &0 &0  &0  \vspace{5pt}\\
0& 0 & -1 &0 &0  &0  \vspace{5pt}\\
0& -1& 0  &0 &0  &2  \endbmatrix\quad Q_1=\bmatrix
 0& 0& 0 & 1 &0  &0 \vspace{5pt}\\
 0& 2& 0 & 0 &0  &1 \vspace{5pt}\\
 0& 0& 0 & 0 &-1 &0 \vspace{5pt}\\
 1& 0& 0 &-1 &0  &0 \vspace{5pt}\\
 0& 0& -1& 0 &0  &0 \vspace{5pt}\\
 0& 1& 0 & 0 &0  &1 
\endbmatrix, \quad \rho_1(\pmb\delta)=1,$$
$$\begin{array}{lll}
(\rho_1\otimes b_1)(1\otimes (y^2)^*)&=&x^3-xy^2,\vspace{5pt}\\
(\rho_1\otimes b_1)(1\otimes (yz)^*)&=&-x^2z,\vspace{5pt}\\
(\rho_1\otimes b_1)(1\otimes (z^2)^*)&=&x^2y+xz^2,\vspace{5pt}\\
(\rho_1\otimes b_1)(1\otimes y^3)&=&y^3, \vspace{5pt}\\
(\rho_1\otimes b_1)(1\otimes y^2z)&=&y^2z, \vspace{5pt}\\
(\rho_1\otimes b_1)(1\otimes yz^2)&=&yz^2 +xy^2, \vspace{5pt}\\
(\rho_1\otimes b_1)(1\otimes z^3)&=&z^3 +2xyz, \ \text{and}\end{array}$$

$$\rho_1\otimes b_2=
\bmatrix 
 0& 0& 0& z&-y& 0& 0\vspace{5pt}\\
 0& 0& 0& x& z&-y& 0\vspace{5pt}\\
 0& 0& 0& 0& x& z&-y\vspace{5pt}\\
-z&-x& 0& 0&-x& 0& 0\vspace{5pt}\\
 y&-z&-x& x& 0& 0& 0\vspace{5pt}\\ 
 0& y&-z& 0& 0& 0& x\vspace{5pt}\\
 0& 0& y& 0& 0&-x& 0\endbmatrix.$$
\end{example}
\begin{example}\label{i=2} When $i=2$, then $\Phi_2=(x^2y^2)^*-(xyz^2)^*+2(z^4)^*$, 

$$T_2=\bmatrix
0& 0 & 0  &1 &0  &0  \vspace{5pt}\\
0& 1 & 0  &0 &0  &-1 \vspace{5pt}\\
0& 0 & 0  &0 &-1 &0  \vspace{5pt}\\
1& 0 & 0  &0 &0  &0  \vspace{5pt}\\
0& 0 & -1 &0 &0  &0  \vspace{5pt}\\
0& -1& 0  &0 &0  &2  \endbmatrix, \quad Q_2=\bmatrix
 0& 0& 0 & 1 &0  &0 \vspace{5pt}\\
 0& 2& 0 & 0 &0  &1 \vspace{5pt}\\
 0& 0& 0 & 0 &-1 &0 \vspace{5pt}\\
 1& 0& 0 & 0 &0  &0 \vspace{5pt}\\
 0& 0& -1& 0 &0  &0 \vspace{5pt}\\
 0& 1& 0 & 0 &0  &1 
\endbmatrix, \quad  \rho_2(\pmb\delta)=1,$$
$$\begin{array}{lll}
(\rho_2\otimes b_1)(1\otimes (y^2)^*)&=&x^3,\vspace{5pt}\\
(\rho_2\otimes b_1)(1\otimes (yz)^*)&=&-x^2z,\vspace{5pt}\\
(\rho_2\otimes b_1)(1\otimes (z^2)^*)&=&x^2y+xz^2,\vspace{5pt}\\
(\rho_2\otimes b_1)(1\otimes y^3)&=&y^3, \vspace{5pt}\\
(\rho_2\otimes b_1)(1\otimes y^2z)&=&y^2z, \vspace{5pt}\\
(\rho_2\otimes b_1)(1\otimes yz^2)&=&yz^2 +xy^2, \vspace{5pt}\\
(\rho_2\otimes b_1)(1\otimes z^3)&=&z^3 +2xyz,\text{ and}\end{array}$$
$$\rho_2\otimes b_2=
\bmatrix 
 0& 0& 0& z&-y& 0& 0\vspace{5pt}\\
 0& 0& 0& x& z&-y& 0\vspace{5pt}\\
 0& 0& 0& 0& x& z&-y\vspace{5pt}\\
-z&-x& 0& 0& 0& 0& 0\vspace{5pt}\\
 y&-z&-x& 0& 0& 0& 0\vspace{5pt}\\ 
 0& y&-z& 0& 0& 0& x\vspace{5pt}\\
 0& 0& y& 0& 0&-x& 0\endbmatrix.$$
\end{example}

\begin{example}\label{6.4} When $i=3$, then $\Phi_3= (y^2z^2)^*+(x^2z^2)^*+(x^2y^2)^*+2(xyz^2)^*+2(xy^2z)^*+2(x^2yz)^*$,
$$T_3 = \left(\begin{array}{cccccc} 0&0&0&1&2&1\vspace{5pt}\\0&1&2&0&2&2\vspace{5pt}\\0&2&1&2&2&0\vspace{5pt}\\1&0&2&0&0&1\vspace{5pt}\\2&2&2&0&1&0\vspace{5pt}\\1&2&0&1&0&0 \end{array} \right),\quad Q_3 = \left(\begin{array}{rrrrrrr} 27&-18&-18&9&18&9\vspace{5pt}\\-18&18&0&-18&0&18\vspace{5pt}\\-18&0&18&18&0&-18\vspace{5pt}\\9&-18&18&27&-18&9\vspace{5pt}\\18&0&0&-18&18&-18\vspace{5pt}\\9&18&-18&9&-18&27 \end{array} \right),$$ $\rho_3(\pmb \delta)=54$,
\begin{equation}\label{x3}\begin{array}{lll}  
(\rho_3\otimes b_1)(1 \otimes (y^2)^*) &=&x(9x^2-18xy+18xz+27y^2-18yz+9z^2),\vspace{5pt}\\
(\rho_3\otimes b_1)(1\otimes (yz)^*)&=&x(18x^2-18y^2+18yz-18z^2),  \vspace{5pt}\\
(\rho_3\otimes b_1)(1\otimes (z^2)^*)&=&x(9x^2+18xy-18xz+9y^2-18yz+27z^2), \vspace{5pt}\\
(\rho_3\otimes b_1)(1\otimes y^3)&=&54y^3,\vspace{5pt}\\
(\rho_3\otimes b_1)(1\otimes y^2z)&=&54y^2z-x(36x^2-36xy-18xz+36y^2+36yz), \vspace{5pt}\\
(\rho_3\otimes b_1)(1\otimes yz^2)&=&54yz^2 -x(36x^2-18xy-36xz+36yz+36z^2),\vspace{5pt}\\
(\rho_3\otimes b_1)(1\otimes z^3)&=&54z^3, \end{array}\end{equation}and $\rho_3\otimes b_2$ is equal to 
$$  \left(\begin{array}{ccccccc}0&-54x&-36x&-36x+54z&-36x-54y&0&0\vspace{5pt}\\ 54x&0&-54x&0&54z&-54y&0\vspace{5pt}\\ 36x&54x&0& 0&0&36x+54z&36x- 54y\vspace{5pt}\\
36x-54z&0&0&0&27x&-18x&9x\vspace{5pt}\\ 36x+54y&-54z&0&-27x&0&9x&-18x\vspace{5pt}\\ 0&54y&-36x-54z&18x&-9x&0&27x\vspace{5pt}\\0&0&-36x+54y&-9x&18x&-27x&0  \end{array} \right). $$
\end{example}

\bigskip
We promised in (\ref{6.0}) to show that $\Phi_3$ from Example~\ref{6.4} is the Macaulay inverse system for the ideal \begin{equation}\label{I3}I_3=(x^3,y^3,z^3):(x+y+z)^2.\end{equation} This promise is fulfilled in the next Lemma. It is clear that   the the right side of (\ref{I3}) contains $3$ linearly independent perfect cubes even though it is  not immediately obvious from the generators of $I_3$ listed in (\ref{x3}) that $x^3$ is in $I_3$. On the other hand, $(\rho_3\otimes b_1)((y^2)^*+2(yz)^*+(z^2)^*)=54x^3$.

\begin{lemma}\label{6.2}The Macaulay inverse system for the ideal $(x^{n}, y^{n}, z^{n}):(x+y+z)^{n-1}$ in the ring $\mathbb Z[x,y,z]$ is
$$ 
\sum_{\begin{array}{cc} a+b+c =n-1 \\ a,b,c \leq n-1 \end{array}}
\binom{n-1}{a, b, c} (x^{n-1-a}y^{n-1-b}z^{n-1-c})^*.$$
\end{lemma}

\begin{proof}
Observe that 
\begin{align*}  f \in (x^{n}, y^{n}, z^{n}):(x+y+z)^{n-1} & \Leftrightarrow f(x+y+z)^{n-1} \in (x^{n}, y^{n}, z^{n}).
 \end{align*}
The Macaulay inverse system of  $(x^{n}, y^{n}, z^{n})$ is $(x^{n-1}y^{n-1}z^{n-1})^*$; therefore,  
$$\begin{array}{lll}
 f \in (x^{n}, y^{n}, z^{n}):(x+y+z)^{n-1} & \Leftrightarrow& f(x+y+z)^{n-1} \in  \operatorname{ann}\left((x^{n-1}y^{n-1}z^{n-1})^*\right) \vspace{5pt}\\
&  \Leftrightarrow & f(x+y+z)^{n-1} \left((x^{n-1}y^{n-1}z^{n-1})^*\right) =0\vspace{5pt}\\
& \Leftrightarrow & f\left((x+y+z)^{n-1} \left((x^{n-1}y^{n-1}z^{n-1})^*\right)\right) =0\vspace{5pt}\\
&  \Leftrightarrow &f\left(  \sum\limits_{\genfrac{}{}{0pt}{}{a+b+c =n-1}{a,b,c \leq n-1}}\binom{n-1} {a, b, c} (x^{n-1-a}y^{n-1-b}z^{n-1-c})^*\right)=0.
 \end{array}$$
\end{proof}

\end{document}